\documentclass[sn-mathphys]{sn-jnl}

\jyear{2021}%

\theoremstyle{thmstyleone}%
\newtheorem{theorem}{Theorem}
\newtheorem{proposition}[theorem]{Proposition}%

\theoremstyle{thmstyletwo}%
\newtheorem{example}{Example}%
\newtheorem{remark}{Remark}%
\newtheorem{lemma}{Lemma}%

\theoremstyle{thmstylethree}%

\begin{document}

\title[Article Title]{Characterization of invariant complex Finsler metrics and Schwarz lemma on the classical domains
}


\author{\fnm{Chunping} \sur{Zhong}}\email{zcp@xmu.edu.cn}



\affil{\orgdiv{School of Mathematical Sciences}, \orgname{Xiamen
University},  \city{Xiamen}, \postcode{361005}, \country{China}}




\abstract{
Our goal of this paper is to give a complete characterization of all holomorphic invariant strongly pseudoconvex complex Finsler metrics on the classical domains and establish a corresponding Schwarz lemma for holomorphic mappings with respect to these invariant metrics. We prove that every $\mbox{Aut}(\mathfrak{D})$-invariant strongly pseudoconvex complex Finsler metric $F$ on a classical domain $\mathfrak{D}$ is a K\"ahler-Berwald metric which is not necessary Hermitian quadratic, but it enjoys very similar curvature property as that of the Bergman metric on $\mathfrak{D}$. In particular, if $F$ is Hermitian quadratic, then $F$  must be a constant multiple of the Bergman metric on $\mathfrak{D}$.
This actually answers the $4$-th open problem posed by Bland and Kalka (Variations of holomorphic curvature for K\"ahler Finsler metrics, American Mathematical Society, 1996).
 We also obtain a general Schwarz lemma for holomorphic mappings from a classical domain $\mathfrak{D}_1$ into another classical domain $\mathfrak{D}_2$ whenever $\mathfrak{D}_1$ and $\mathfrak{D}_2$ are endowed with arbitrary holomorphic invariant K\"ahler-Berwald metrics $F_1$ and $F_2$, respectively. The method used to prove the Schwarz lemma is purely geometric. Our results show that the Lu constant of $(\mathfrak{D},F)$ is both an analytic invariant and a geometric invariant. This can be better understood in the complex Finsler setting.
}




\maketitle

\section{Introduction and main results}

In order to classify domains in the sense of biholomorphic equivalence in high dimensional cases, it is very natural to introduce holomorphic invariant metrics. In 1922, S. Bergman introduced the Bergman kernel and metric  which have been a seminal part of geometric analysis and partial differential equation in several complex variables ever since then (cf. \cite{Krantz}). Note that the Bergman metric is a Hermitian quadratic metric.  In the geometric function theory of several complex variables, however, the most often used invariant metrics (such as the Carath\'eodory metric and the Kobayashi metric) are often complex Finsler metrics and  in general they do not have enough regularity to admit differential geometric studies.

Very recently in \cite{Zh2}, the author of the present paper proved that on the open unit ball $B_n$ in $\mathbb{C}^n$ there exists no $\mbox{Aut}(B_n)$-invariant strongly pseudoconvex complex Finsler metrics other than a constant multiple of the Bergman metric; on the unit polydisk $P_n$ in $\mathbb{C}^n$ with $n\geq 2$, however, the author explicitly constructed infinitely many $\mbox{Aut}(P_n)$-invariant strongly convex K\"ahler-Berwald metrics which are neither Hermitian quadratic nor holomorphically isometric.

Note that $B_n$ is an irreducible bounded symmetric domain with rank one, while $P_n$ is a reducible bounded symmetric domain with rank $n$. It is natural to ask whether this phenomenon also happens on a classical domain (cf. Hua \cite{Hua}). This question was  affirmatively answered in \cite{GZ}. The metrics constructed in \cite{GZ} are of particular interest since they are the only holomorphically invariant complex Finsler metrics which are not Hermitian quadratic found on classical domains. They are explicitly constructed, have good regularity properties and are strongly pseudoconvex. Consequently they are allowed differential geometrical applications. In this paper, we shall see that these metrics also enjoy curvature properties very similar to those of the Bergman metrics.
The results in \cite{GZ} also show that as far as  a classical domain with  rank $\geq 2$ is concerned, other than the geometry of Bergman metrics, it makes sense to develop the differential geometry of holomorphically invariant complex Finsler metrics which are not Hermitian quadratic.

For higher rank classical domains it is natural to look for the characterization of holomorphically invariant strongly pseudoconvex complex Finsler metrics and to ask whether such metrics are necessarily K\"ahler-Finsler in the sense of \cite{AP}.
As our first main result we give a complete characterization of such kinds of metrics on the classical domains  and show that indeed they are all K\"ahler-Finsler metrics (cf. Theorem \ref{thm-2} and Theorem \ref{thm-7}).

\begin{theorem}\label{m-th-1}
Let $\mathfrak{D}$ be a classical domain and $F:T^{1,0}\mathfrak{D}\rightarrow[0,+\infty)$ be an $\mbox{Aut}(\mathfrak{D})$-invariant strongly pseudoconvex complex Finsler metric. Then the following assertions are true:

(1) $F$ is necessary a  K\"ahler-Finsler metric. More precisely, $F$ is a complete K\"ahler-Berwald metric.

(2) The  holomorphic sectional curvature of $F$ is bounded from below and above by negative constants and the holomorphic bisectional curvature of $F$ is non-positive and bounded below by a negative constant.
\end{theorem}
\begin{remark}
If $\mathfrak{D}=\mathfrak{R}_I(1,n)$, namely the classical domain of type $I$ with rank$(\mathfrak{D})=1$, then $\mathfrak{D}$ reduces to the open unit ball in $\mathbb{C}^n$. In this case, $\mathfrak{D}$ admits no $\mbox{Aut}(\mathfrak{D})$-invariant strongly pseudoconvex complex Finsler metric other than  a constant multiple of the Bergman metric \cite{Zh2}. If $\mbox{rank}(\mathfrak{D})\geq 2$, then  $\mathfrak{D}$ admits infinitely many $\mbox{Aut}(\mathfrak{D})$-invariant strongly pseudoconvex complex Finsler metrics which are neither Hermitian quadratic nor holomorphically isometric \cite{GZ}.
\end{remark}

The second purpose of this paper is to establish a general Schwarz lemma for holomorphic mappings between two classical domains whenever they are endowed with arbitrary holomorphically invariant K\"ahler-Berwald metrics which are not necessarily Hermitian quadratic. It is known that various kinds of Schwarz lemmas have played a very important role both in complex analysis and complex differential geometry.  Many researchers have contributed to the study of Schwarz lemma and striking results have been discovered on K\"ahler manifolds or more general Hermitian manifolds with suitable curvature assumptions. We cannot mention all related work but only name a few work here: Ahlfors \cite{Afs}, Look \cite{Look3}, Kobayashi \cite{Kobayashi}, Royden \cite{Royden}, Yau \cite{Yau}. See the book by Kim and Lee \cite{KL} and references therein for a detailed account on this topic from the viewpoint of differential geometry. The study of Schwarz lemma on complex Finsler manifolds is very recent. We refer the work of  \cite{SS}, \cite{Wan},\cite{NZ2}, \cite{LZ}.
 The study of Schwarz lemma on a classical domain $\mathfrak{D}$ endowed with the Bergman metric $ds_{\mathfrak{D}}^2$ goes back to the work of Look \cite{Look3}, where he proved case-by-case that for any holomorphic mapping $f$ from $\mathfrak{D}$ into itself,
 \begin{equation}
 f^\ast ds_{\mathfrak{D}}^2\leq l_0^2(\mathfrak{D})ds_{\mathfrak{D}}^2.\label{Lu}
 \end{equation}
Look also proved that $l_0^2(\mathfrak{D})=\mbox{rank}(\mathfrak{D})$, hence it is an analytic invariant of $\mathfrak{D}$. Note that the constant $l_0(\mathfrak{D})$ is optimal in the sense that it can also be achieved and is the the least constant such that \eqref{Lu} is satisfied. This fact was first observed and pointed out by Look \cite{Look3}. We call $l_0(\mathfrak{D})$ the Lu constant with respect to the Bergman metric on $\mathfrak{D}$. Indeed, $l_0(\mathfrak{D})=\sqrt{\frac{K_1}{K_2}}$, where $-K_1<0$ and $-K_2<0$ are the lower and upper bounds of the holomorphic sectional curvature of the Bergman metric  on $\mathfrak{D}$, respectively (cf. \cite{Look3}).

As our second main result we prove:
\begin{theorem}\label{m-th-2}
Let $\mathfrak{D}_1$ be a classical domain which is endowed with an arbitrary $\mbox{Aut}(\mathfrak{D}_1)$-invariant  K\"ahler-Berwald metric $F_1$ such that its holomorphic sectional curvature is bounded below by a negative constant $-K_1$. Let $\mathfrak{D}_2$ be another classical domain which is endowed with an $\mbox{Aut}(\mathfrak{D}_2)$-invariant K\"ahler-Berwald metric $F_2$ such that its holomorphic sectional curvature is bounded above by a negative constant $-K_2$. Then for every holomorphic mapping $f$ from $\mathfrak{D}_{1}$ into $\mathfrak{D}_{2}$, we have
\begin{equation}
(f^\ast F_2)(Z;V)\leq\sqrt{\frac{K_1}{K_2}}F_1(Z;V),\quad \forall (Z;V)\in T^{1,0}\mathfrak{D}_{1}.\label{fhg}
\end{equation}
\end{theorem}
Comparing with the Schwarz lemma obtained in \cite{NZ2}, we do not require the metrics $F_1$ and $F_2$ in Theorem \ref{m-th-2}  to be strongly convex, hence does not need the assumption on the flag curvature (which is a real geometric quantity) of $F_1$ and $F_2$, respectively.  Comparing with the case of Bergman metrics \cite{Look3}, in order to obtain Theorem \ref{m-th-2}, we need to overcome several difficulties. The hardest one is the fact that in the non-Hermitian quadratic case there are infinitely many choices of $F_1$ on $\mathfrak{D}_1$ and $F_2$ on $\mathfrak{D}_2$, respectively. The method in Look \cite{Look3} cannot be directly used to prove the corresponding Schwarz lemma whenever the invariant metrics are non-Hermitian quadratic. The method that we use to  prove Theorem \ref{m-th-2} is purely geometric and it is essentially inspired by Kobayashi \cite{Kobayashi}. In the non-Hermitian quadratic case, however, we need to innovate Kobayashi's idea in order to be effective. The key point to prove Theorem \ref{m-th-2} is to establish  Theorem \ref{thm-fcc} and then use it to obtain Theorem \ref{lemb}.

Taking $\mathfrak{D}_1=\mathfrak{D}_2$ and  $F_1=F_2$ in Theorem \ref{m-th-2} and using Theorem \ref{m-th-1}, we immediately obtain the following theorem.
\begin{theorem}\label{m-th-3}
Let $\mathfrak{D}$ be a classical domain which is endowed with an $\mbox{Aut}(\mathfrak{D})$-invariant  K\"ahler-Berwald metric $F$ such that its holomorphic sectional curvature is bounded below and above by negative constants $-K_1$ and $-K_2$, respectively. Then for every holomorphic mapping $f$ from $\mathfrak{D}$ into itself, we have
\begin{equation}
(f^\ast F)(Z;V)\leq \sqrt{\frac{K_1}{K_2}}F(Z;V),\quad \forall(Z;V)\in T^{1,0}\mathfrak{D}.\label{sc}
\end{equation}
\end{theorem}

We call the constant $l_0(\mathfrak{D},F):=\sqrt{\frac{K_1}{K_2}}$ in \eqref{sc}  the Lu constant of $(\mathfrak{D},F)$, which is also optimal in the sense that it is the least constant such that \eqref{sc} is satisfied  and can be achieved.
\begin{example}[\cite{GZ}]\label{example}
		Let $\mathfrak{D}=\mathfrak{R}_I(m,n)$, for any fixed positive integer $k (\geq 2)$ and any fixed $t\in[0,+\infty)$, define
 \begin{eqnarray*}
			F_{t,k}^2(Z;V)&:=&\frac{m+n}{1+t}\Bigg\{\mbox{tr}\left\{(I-ZZ^\ast)^{-1}V(I-Z^\ast Z)^{-1}V^\ast\right\}\\
			&&+t\sqrt[k]{\mbox{tr}\left\{\left[(I-ZZ^\ast)^{-1}V(I-Z^\ast Z)^{-1}V^\ast\right]^k\right\}}\Bigg\},\quad \forall(Z; V)\in T^{1,0}\mathfrak{R}_I.
	\end{eqnarray*}
Then $F_{t,k}$ is an $\mbox{Aut}(\mathfrak{R}_I)$-invariant K\"ahler-Berwald metric on $\mathfrak{R}_I$. The lower and upper bounds of the holomorphic sectional curvature of $F_{t,k}$ was given by Theorem 3.26 in \cite{GZ}. Thus by the above Theorem \ref{m-th-3}, for any holomorphic mapping $f$ from $\mathfrak{R}_I$ into itself, we have
		\begin{eqnarray*}
			(f^\ast F_{t,k}^2)(Z;V)\leq\frac{m+t\sqrt[k]{m}}{1+t}F_{t,k}^2(Z;V),\quad (Z;V)\in T^{1,0}\mathfrak{R}_I.
		\end{eqnarray*}
For $m=1$,  $\mathfrak{R}_I(1,n)=B_n$ is the open unit ball in $\mathbb{C}^n$,  $F_{t,k}$ reduces to the Bergman metric $F_B$ on $B_n$ and $l_0(\mathfrak{R}_I, F_B)=1$.
For $m\geq 2$ and $t=0$,  $F_{0,k}$ reduces to the Bergman metric $F_B$ on $\mathfrak{R}_I$ and $l_0(\mathfrak{R}_I, F_B)=\sqrt{m}$ which was obtained by Look \cite{Look3}.
For $m\geq 2$ and $t\in (0,+\infty)$, $F_{t,k}$ is a non-Hermitian quadratic $\mbox{Aut}(\mathfrak{R}_I)$-invariant K\"ahler-Berwald metric on $\mathfrak{R}_I$ and we have
$$1<l_0(\mathfrak{R}_I, F)=\sqrt{\frac{m+t\sqrt[k]{m}}{1+t}}<\sqrt{\mbox{rank}(\mathfrak{R}_I)}.$$
\end{example}

Theorem \ref{m-th-3} and Example \ref{example} show that the Lu constant of $(\mathfrak{D},F)$ is both an analytic invariant and a geometric invariant. This can be better understood in complex Finsler setting since by Theorem \ref{m-th-1}, (up to multiplicative constants) there is only one $\mbox{Aut}(\mathfrak{D})$-invariant Hermitian quadratic metric (namely the Bergman metric) on $\mathfrak{D}$ and a constant multiple of the Bergman metric does not change the Lu constant, namely  $l_0(\mathfrak{D},F)=l_0(\mathfrak{D},\lambda F)$ for any $\lambda\in\mathbb{R}$ with $\lambda>0$, in this case we always have (cf. \cite{Look3})
$$
l_0(\mathfrak{D},F)=\sqrt{\frac{K_1}{K_2}}\equiv \sqrt{\mbox{rank}(\mathfrak{D})}.
$$
Whenever $\mbox{rank}(\mathfrak{D})\geq 2$, however,  $\mathfrak{D}$ admits infinitely many $\mbox{Aut}(\mathfrak{D})$-invariant non-Hermitian quadratic K\"ahler-Berwald metrics, which are not holomorphically isometric among them (cf. \cite{GZ}). In this case,
 $$
 l_0(\mathfrak{D},F)=\sqrt{\frac{K_1}{K_2}} \not\equiv \sqrt{\mbox{rank}(\mathfrak{D})},
 $$
namely the Lu constant $l_0(\mathfrak{D},F)$ depends on the choice of $F$ on $\mathfrak{D}$.
In both cases, we have $l_0(\mathfrak{D},F)\geq 1$, and the equality holds if and only if the holomorphic sectional curvature of $F$ is equal to a negative constant.

The rest of this paper will be organized as follows. In Section \ref{section-2}, we recall some basic notions of complex Finsler geometry on domains in $\mathbb{C}^n$. In Section \ref{section-3}, after giving a characterization of $\mbox{Iso}(\mathfrak{D})$-invariant complex Minkowski norm $\pmb{f}$ on  $T_0^{1,0}\mathfrak{D}$ for $\mathfrak{D}=\mathfrak{R}_I,\mathfrak{R}_{II},\mathfrak{R}_{III}$ and $\mathfrak{R}_{IV}$, respectively. We give the proof of Theorem \ref{m-th-1} by proving Theorem \ref{thm-2} for $\mathfrak{D}=\mathfrak{R}_I,\mathfrak{R}_{II},\mathfrak{R}_{III}$ and Theorem \ref{thm-7} for $\mathfrak{D}=\mathfrak{R}_{IV}$, respectively. In Section \ref{section-4}, we first establish a comparison theorem  of continuous complex Finsler metrics with the Carath\'eodory metric on a convex domains (cf. Theorem \ref{thm-fcc}),  as an application, we obtain a comparison theorem of holomorphically invariant strongly pseudoconvex  complex Finsler metrics with the Carath\'eodory metric on the classical domains and then use it to give the proof of Theorem \ref{m-th-2}.

\section{Complex Finsler metrics on domains in $\mathbb{C}^n$}\label{section-2}

In this section, we recall some necessary notations of complex Finsler geometry on domains in $\mathbb{C}^n$ with $n\geq 2$.  We refer to  \cite{AP} for more details of real and complex Finsler geometry on  manifolds.

Let $D$ be a domain in $\mathbb{C}^n$ and $T^{1,0}D\cong D\times \mathbb{C}^n$ the holomorphic tangent bundle of $D$. We denote $\widetilde{T^{1,0}D}$ the complement of the zero section in $T^{1,0}D$, and $\widetilde{T_p^{1,0}D}$ the complement of the zero vector in $T_p^{1,0}D$.
A (smooth) complex Finsler metric on $D$ is a continuous function $F:T^{1,0}D\rightarrow [0,+\infty)$ satisfying $F(z;\lambda v)=\mid\lambda\mid F(z;v)$
	for any $(z;v)\in T^{1,0}D$ and $\lambda\in\mathbb{C}$ and $G:=F^2$ is $C^\infty$ on
	$\widetilde{T^{1,0}D}$. Note that the smoothness is required only on $\widetilde{T^{1,0}D}$ but not on the whole $T^{1,0}D$, since $F$ is smooth on $T^{1,0}D$ if and only if $F$ is a Hermitian metric \cite{AP}.  In the following a complex Finsler metric is always assumed to be $C^\infty$ in the above sense until otherwise stated.
	A complex Finsler metric $F$ is called strongly pseudoconvex if the following Hermitian matrix
	$(G_{i\overline{j}}):=\left(\frac{\partial^2G}{\partial v_i\partial\overline{v_j}}\right)$
	is positive definite at any point $(z;v)\in \widetilde{T^{1,0}D}$ (cf. \cite{AP}).
 Note that for each fixed point $z_0\in D$, a strongly pseudoconvex complex Finsler metric $F$ on $D$ induces  a strongly pseudoconvex complex norm $f_{z_0}(v):=F(z_0;v)$ on $T_{z_0}^{1,0}D\cong\{z_0\}\times\mathbb{C}^n$, which is called a complex Minkowski norm. If, moreover, $f_{z_0}(v)$ is the same for any $z_0\in D$, then $F$ is called a complex Minkowski metric on $D$ (cf. \cite{A2}).
	
For a strongly pseudoconvex complex Finsler metric $F:T^{1,0}D\rightarrow [0,+\infty)$, if $F^2=\sum_{i,j=1}^ng_{i\overline{j}}(z)v_i\overline{v_j}$, namely $\frac{\partial^2F^2}{\partial v_i\partial\overline{v_j}}$ are independent of $v$, then $F$ is called Hermitian quadratic; otherwise $F$ is called non-Hermitian quadratic. For example, every Hermitian metric on $D$ is Hermitian quadratic.

Denote $(G^{\overline{j}l})$ the inverse matrix of $(G_{i\overline{j}})$ such that
	$\sum_{j=1}^nG_{i\overline{j}}G^{\overline{j}l}=\delta_{i}^l$, where $\delta_{i}^l=1$ for $i=l$ and $\delta_i^l=0$ for $i\neq l$.
	Then the Chern-Finsler non-linear connection coefficients $\varGamma_{;i}^l$ associated to $F$ are given by $\varGamma_{;i}^l=\displaystyle\sum_{s=1}^nG^{\overline{s}l}\frac{\partial^2G}{\partial z_i\partial\overline{v_s}}$.
If $F$ is Hermitian quadratic, then the coefficients $\varGamma_{;i}^l(z;v)$ are complex-linear with respect to $v$, otherwise $\varGamma_{;i}^l(z;v)$ are not necessary complex-linear with respect to $v$.
	The horizontal Chern-Finsler connection coefficients $\varGamma_{j;i}^l(z;v)$ associated to $F$ are given by
$\varGamma_{j;i}^l(z;v)=\frac{\partial \varGamma_{;i}^l}{\partial v_j}$, which satisfy $\displaystyle\sum_{j=1}^n\varGamma_{j;i}^l(z;v) v_j=\varGamma_{;i}^l(z;v)$.
	Using the homogeneity of $F$, it is easy to check that
	$\varGamma_{;i}^l(z;\lambda v)=\lambda\varGamma_{;i}^l(z;v)$ and $\varGamma_{j;i}^l(z;\lambda v)=\varGamma_{j;i}^l(z;v)$ for any $\lambda\in\mathbb{C}^\ast:=\mathbb{C}\setminus\{0\}$.
	$F$ is called a K\"ahler-Finsler metric if
$\varGamma_{j;i}^l(z;v)=\varGamma_{i;j}^l(z;v)$, which is equivalent to the condition $\varGamma_{;i}^l(z;v)=\displaystyle\sum_{j=1}^n\varGamma_{i;j}^l(z;v) v_j$ (cf. \cite{CS}); $F$ is called a complex Berwald metric or a complex manifold modeled a complex Minkowski space (cf. \cite{A2}) if
	$\varGamma_{j;i}^l(z;v)$ are locally independent of $v$, ;  $F$ is called a K\"ahler-Berwald metric if $F$ is both a K\"ahler-Finsler metric and a complex Berwald metric.  By definition, if $F$ is a Hermitian metric, then $F$ is necessary a complex Berwald metric. Thus every K\"ahler metric is obvious a  (Hermitian quadratic) K\"ahler-Berwald metric. There are, however, lots of non-Hermitian quadratic K\"ahler-Berwald metrics (cf. \cite{Zh2}).

Let $D$ be a domain in $\mathbb{C}^n$ and $F:T^{1,0}D\rightarrow [0,+\infty)$ a complex Finsler metric. Denote $\mbox{Aut}(D)$ the holomorphic automorphism group of $D$, then $F$ is called holomorphically invariant if
$
F(f(z);f_\ast(v))=F(z;v)
$
holds for any $f\in\mbox{Aut}(D)$ and any $(z;v)\in T^{1,0}D$, here $f_\ast$ denotes the differential of $f$ at the point $z\in D$. If furthermore $D$ is a homogeneous domain which contains the origin $0\in D$, we denote $\mbox{Iso}(D)$ the isotropy subgroup  at the origin.

\section{Characterization of invariant complex Finsler metrics on the classical domains}\label{section-3}

Denote $\mathscr{M}(m,n)$ the set of all $m\times n$ matrices over $\mathbb{C}$, which is a complex vector space isomorphic to $\mathbb{C}^{m n}$. In this paper, a point in $\mathbb{C}^n$ is always understood as an element in $\mathscr{M}(1,n)$.
	If $m=n$, we denote $\mathscr{M}(n)$ instead of $\mathscr{M}(n,n)$.
	Let $Z\in\mathscr{M}(m,n)$, we denote $Z'$ the transpose of $Z$, $\overline{Z}$ the complex conjugate of $Z$, and $Z^\ast=\overline{Z'}$ the conjugate transpose of $Z$. If $A\in \mathscr{M}(n)$ is a Hermitian matrix, we also denote $A>0$ to mean
	$A$ is positive definite, and $A\geq 0$ to mean $A$ is semi-positive definite.

A complex manifold $M$ with a Hermitian metric $h$ is said to be a Hermitian symmetric space if for every point $p\in M$, there exists an involutive holomorphic isometry $\sigma_p$ of $(M,h)$ such that $p$ is an isolated fixed point. It is known that $(M,h)$ is necessary a K\"ahler manifold (cf. \cite{Helgason}). An irreducible Hermitian symmetric space of non-compact type can be realized as a bounded symmetric domain in some complex Euclidean space $\mathbb{C}^N$ by the Harish-Chandra embedding (cf. \cite{Helgason}). \'E. Cartan \cite{Cartan} proved that, if $N\neq 16, 27$, every irreducible bounded symmetric domain $\mathfrak{D}$ in $\mathbb{C}^N$ is bihomorphically equivalent to one of the following classical domains (cf. Hua \cite{Hua}):
	\begin{eqnarray*}
		\mathfrak{R}_I(m,n)&=&\left\{Z\in \mathscr{M}(m,n): I_m-ZZ^\ast>0\right\}\subset\mathbb{C}^{mn},\quad (m\leq n),\\
		\mathfrak{R}_{II}(m)&=&\left\{Z\in\mathscr{M}(m): I_m-ZZ^\ast>0,\; Z=Z'\right\}\subset\mathbb{C}^{\frac{m(m+1)}{2}},\\
		\mathfrak{R}_{III}(m)&=&\left\{ Z\in\mathscr{M}(m): I_m-ZZ^\ast>0,\; Z=-Z'\right\}\subset\mathbb{C}^{\frac{m(m-1)}{2}},\\
		\mathfrak{R}_{IV}(N) &=&\left\{z\in\mathbb{C}^N: 1+\mid zz'\mid^2-2zz^\ast>0,\;1-\mid zz'\mid>0\right\},\quad (N\geq 2)
	\end{eqnarray*}
	where $I_m$ denotes the identity matrix of order $m$. The classical domains are all complete circular, homogeneous, convex domains with the origin being the center.
The rank of $\mathfrak{R}_I(m,n), \mathfrak{R}_{II}(m), \mathfrak{R}_{III}(m)$ and $\mathfrak{R}_{IV}(N)$ are $m, m,[m/2]$ and $2$, respectively.

 The purpose of this section is to give a complete characterization of $\mbox{Aut}(\mathfrak{D})$-invariant strongly pseudoconvex complex Finsler metrics on every one  of the above classical domains $\mathfrak{D}$.

\subsection{The cases of $\mathfrak{R}_I,\mathfrak{R}_{II},\mathfrak{R}_{III}$}

In this section, we give a characterization of $\mbox{Aut}(\mathfrak{D})$-invariant complex Finsler metrics on the classical domains $\mathfrak{R}_I,\mathfrak{R}_{II},\mathfrak{R}_{III}$. Let's first recall its holomorphic automorphism groups, we refer to Hua \cite{Hua} and Lu \cite{Lu} for more details.

The automorphism of $\mathfrak{R}_I(m,n)$ which transforms a fixed point $Z_0\in\mathfrak{R}_I$ into the origin is given in the form (cf. Hua \cite{Hua})
\begin{equation}
\varPhi_{Z_0}(Z)=A(Z-Z_0)\left(I_n-Z_0^\ast Z\right)^{-1}D^{-1},\quad Z\in\mathfrak{R}_I,\label{p-I}
\end{equation}
where $A\in\mathscr{M}(m)$ and $D\in\mathscr{M}(n)$ satisfy
\begin{equation}
A^\ast A=\left(I_m-Z_0Z_0^\ast\right)^{-1},\quad D^\ast D=\left(I_n-Z_0^\ast Z_0\right)^{-1}.\label{p-Ia}
\end{equation}
The differential $(\varPhi_{Z_0})_\ast$ of $\varPhi_{Z_0}$ at the point $Z_0\in\mathfrak{R}_I$ can be expressed (in matrix form) as follows:
\begin{equation}
(\varPhi_{Z_0})_\ast (V)=AV\left(I_n-Z_0^\ast Z_0\right)^{-1}D^{-1},\quad \forall V\in T_{Z_0}^{1,0}\mathfrak{R}_I,\label{D-I}
\end{equation}
where $A\in\mathscr{M}(m)$ and $D\in\mathscr{M}(n)$ satisfy \eqref{p-Ia}.

The automorphism of $\mathfrak{R}_{II}(m)$ which transforms a fixed point  $Z_0\in\mathfrak{R}_{II}$ into the origin is given in the form (cf. Hua \cite{Hua})
	\begin{equation}
	\varPhi_{Z_0}(Z)=A(Z-Z_0)(I_m-Z_0^\ast Z)^{-1}{\overline{A}}^{-1},\quad Z\in\mathfrak{R}_{II},\label{p-II}
	\end{equation}
	where$A\in\mathscr{M}(m)$ satisfies
	\begin{equation}
		A^\ast A=(I_m-Z_0Z_0^\ast)^{-1}.\label{p-IIa}
	\end{equation}
 The differential of $\varPhi_{Z_0}$ at the point $Z_0$ is given by
\begin{equation}
(\varPhi_{Z_0})_\ast(V)=AV(I_m-Z_0^\ast Z_0)^{-1}{\overline{A}}^{-1},\quad V\in T_{Z_0}^{1,0}\mathfrak{R}_{II},\label{D-II}
\end{equation}
where $A\in\mathscr{M}(m)$ satisfies \eqref{p-IIa}.

The automorphism of $\mathfrak{R}_{III}(m)$  which transform a fixed point $Z_0\in\mathfrak{R}_{III}$ into the origin is given in the form (cf. Hua \cite{Hua})
	\begin{equation}
		\varPhi_{Z_0}(Z)=A(Z-Z_0)(I_m-Z_0^\ast Z)^{-1}{\overline{A}}^{-1},\quad Z\in\mathfrak{R}_{III},\label{p-III}
	\end{equation}
	where $A\in\mathscr{M}(m)$ satisfies
	\begin{equation}
		A^\ast A=(I_m-Z_0Z_0^\ast)^{-1}.\label{p-IIIa}
	\end{equation}
The differential of $\varPhi_{Z_0}$ at the point $Z_0$ is given by
\begin{equation}
(\varPhi_{Z_0})_\ast(V)=AV(I_m-Z_0^\ast Z_0)^{-1}{\overline{A}}^{-1},\quad V\in T_{Z_0}^{1,0}\mathfrak{R}_{III},\label{D-III}
\end{equation}
where $A\in\mathscr{M}(m)$ satisfies \eqref{p-IIIa}.

The following theorem gives a characterization of $\mbox{Iso}(\mathfrak{D})$-invariant complex norm which is continuous on $T_0^{1,0}\mathfrak{D}$ and smooth on $\widetilde{T_0^{1,0}\mathfrak{D}}$.
\begin{theorem}\label{thm-1}
Let $\pmb{f}:T_{0}^{1,0}\mathfrak{D}\rightarrow [0,+\infty)$ be an $\mbox{Iso}(\mathfrak{D})$-invariant complex norm which is continuous on $T_0^{1,0}\mathfrak{D}$  and smooth on $\widetilde{T_{0}^{1,0}\mathfrak{D}}$ for  $\mathfrak{D}=\mathfrak{R}_I(m,n)$, or $\mathfrak{R}_{II}(m)$, or $\mathfrak{R}_{III}(m)$.
Then there exists a positive integer $k\in\mathbb{N}^+$ and a continuous $1$-homogeneous function $g: [0,+\infty)^k\rightarrow [0,+\infty)$ which is smooth and positive on $(0,+\infty)^k$  such that
\begin{equation}
\pmb{f}^2(V)=g(h_1(V),\cdots,h_{k}(V))
\end{equation}
with
\begin{equation}
h_\alpha(V):=\sqrt[\alpha]{\mbox{tr}\left\{(VV^\ast)^\alpha\right\}},\quad\forall V\in T_0^{1,0}\mathfrak{D}\quad\mbox{for}\quad \alpha=1,\cdots,k. \label{hv}
\end{equation}
\end{theorem}

\begin{proof}
A complex norm  $\pmb{f}:T_0^{1,0}\mathfrak{D}\rightarrow [0,+\infty)$ is $\mbox{Iso}(\mathfrak{D})$-invariant if and only if
$$\pmb{f}(V)=\pmb{f}(\varPhi_\ast(V)),\quad \forall V\in T_{0}^{1,0}\mathfrak{D},\forall \varPhi\in\mbox{Iso}(\mathfrak{D}),$$
where $\varPhi_\ast$ denotes the differential of $\varPhi$ at the origin.  It follows from \eqref{p-I}, \eqref{p-II} and \eqref{p-III} that
\begin{eqnarray*}
\mbox{Iso}(\mathfrak{D})=\left\{
                           \begin{array}{ll}
                             \{\varPhi(Z)=AZD^\ast: Z\in\mathfrak{D}, A\in U(m),D\in  U(n)\}, &  \mathfrak{D}=\mathfrak{R}_I, \\
                             \{\varPhi(Z)=AZA': Z\in\mathfrak{D},A\in U(m)\}, &\mathfrak{D}=\mathfrak{R}_{II}, \\
                             \{\varPhi(Z)=AZA': Z\in\mathfrak{D},A\in U(m)\}, & \mathfrak{D}=\mathfrak{R}_{III}.
                           \end{array}
                         \right.
\end{eqnarray*}
Thus by \eqref{D-I},\eqref{D-II} and \eqref{D-III},  $\pmb{f}$ is $\mbox{Iso}(\mathfrak{D})$-invariant if and only if for any $V\in T_0^{1,0}\mathfrak{D}$,  we have
\begin{eqnarray}
\pmb{f}(V)=\left\{
             \begin{array}{ll}
               \pmb{f}(BVC), \quad B\in U(m),C\in U(n),&  \mathfrak{D}=\mathfrak{R}_I,\\
               \pmb{f}(AVA'), \quad A\in U(m),&\mathfrak{D}=\mathfrak{R}_{II}, \\
               \pmb{f}(AVA'), \quad A\in U(m),&\mathfrak{D}=\mathfrak{R}_{III}.
             \end{array}
           \right.
\label{ifv}
\end{eqnarray}

By Theorem 1.1 on page 299 in \cite{Lu}, for each $V\in T_0^{1,0}\mathfrak{R}_I$ (which is a complex $m$-by-$n$ matrix)  there exists a unitary matrix $B\in U(m)$ and a unitary matrix $C\in U(n)$ such that
\begin{equation}
V=B\begin{pmatrix}
      \lambda_1(V) & \cdots & 0 & 0 & \cdots & 0 \\
      \vdots & \ddots & \vdots & \vdots & \ddots & \vdots \\
      0 & \cdots & \lambda_m(V) & 0 & \cdots & 0 \\
    \end{pmatrix}C,\quad (\lambda_1\geq \cdots\geq \lambda_m\geq 0),\label{VI}
\end{equation}
where $\lambda_1^2(V),\cdots,\lambda_m^2(V)$ are the eigenvalues of the $m$-by-$m$ Hermitian matrix $VV^\ast$.
By Theorem 1.2 on page 301 in \cite{Lu}, for each $V\in T_0^{1,0}\mathfrak{R}_{II}$(which is a complex symmetric matrix of order $m$) there exists a unitary matrix $A\in U(m)$ such that
\begin{equation}
V=A\begin{pmatrix}
     \lambda_1(V) & 0 & \cdots & 0 \\
     0 & \lambda_2(V) & \cdots & 0 \\
     \vdots & \vdots & \vdots & \vdots \\
     0 & 0 & \cdots & \lambda_m(V) \\
   \end{pmatrix}A',\quad (\lambda_1\geq \lambda_2\geq \cdots\geq \lambda_m\geq 0),\label{VII}
\end{equation}
where $\lambda_1^2(V),\cdots,\lambda_m^2(V)$ are the eigenvalues of the $m$-by-$m$ Hermitian matrix $VV^\ast$.
By Theorem 1.3 on page 303 in \cite{Lu}, for each $V\in T_0^{1,0}\mathfrak{R}_{III}$ (which is a complex skew-symmetric matrix of order $m$) there exists a unitary matrix $A\in U(m)$ such that
\begin{equation}
V=\left\{
			\begin{array}{ll}
				A\left[\begin{pmatrix}
					0 & \lambda_1 \\
					-\lambda_1 & 0 \\
				\end{pmatrix}
				\dot{+}\cdots\dot{+}
				\begin{pmatrix}
					0 & \lambda_\nu \\
					-\lambda_\nu & 0 \\
				\end{pmatrix}\right]A', & m=2\nu \\
				A\left[\begin{pmatrix}
					0 & \lambda_1 \\
					-\lambda_1 & 0 \\
				\end{pmatrix}
				\dot{+}\cdots\dot{+}
				\begin{pmatrix}
					0 & \lambda_\nu \\
					-\lambda_\nu & 0 \\
				\end{pmatrix}\dot{+}0\right]A',& m=2\nu+1,
			\end{array}
			\right.
\label{VIII}
\end{equation}
where $\lambda_1^2(V)\geq \cdots\geq \lambda_{[\frac{m}{2}]}^2(V)\geq 0$ are the eigenvalues of the $m$-by-$m$ Hermitian matrix $VV^\ast$ and $\dotplus$ denotes the direct sum of matrices.

It follows from \eqref{ifv}-\eqref{VIII} that there  exists a positive integer $\widetilde{k}$ and a continuous function
$$ \widetilde{g}:[0,+\infty)^{\widetilde{k}}=\underbrace{[0,+\infty)\times \cdots\times[0,+\infty)}_{\widetilde{k}}\rightarrow[0,+\infty)$$ such that
$$
\pmb{f}^2(V)=\widetilde{g}\left(\lambda_1^2(V),\cdots,\lambda_{\widetilde{k}}^2(V)\right),\quad \forall V\in T_0^{1,0}\mathfrak{D}
$$
where
$$
\widetilde{k}=\left\{
             \begin{array}{ll}
               m, & \hbox{for}\;A=I, II,\\
               \left[\frac{m}{2}\right], & \hbox{for}\;A=III.
             \end{array}
           \right.
$$
Note that $\widetilde{g}$ is a symmetric function of $\lambda_1^2(V), \lambda_2^2(V),\cdots,\lambda_{\widetilde{k}}^2(V)$: namely
 \begin{equation}
\widetilde{g}\left(\lambda_{\tau(1)}^2(V),\lambda_{\tau(2)}^2(V),\cdots,\lambda_{\tau(\widetilde{k})}^2(V)\right)=\tilde{g}\left(\lambda_1^2(V),\lambda_2^2(V),\cdots,\lambda_{\widetilde{k}}^2(V)\right)\label{h-g}
\end{equation}
for any permutation $\tau$ of $\left\{1,2,\cdots,\widetilde{k}\right\}$.
Thus $\widetilde{g}$ is a composite function of some symmetric polynomials $\widetilde{\phi}_1,\cdots,\widetilde{\phi}_\ell$ of $\lambda_1^2(V),\cdots,\lambda_{\widetilde{k}}^2(V)$ for some $\ell\in\mathbb{N}^+$, therefore $\widetilde{\phi}_1,\cdots,\widetilde{\phi}_\ell$ can also be expressed as polynomials of  the following elementary symmetric polynomials:
$$
\sigma_1:=\sum_{1\leq i\leq \widetilde{k}}\lambda_i^2(V),\quad\sigma_2:=\sum_{1\leq i<j\leq \widetilde{k}}\lambda_i^2(V)\lambda_j^2(V),\cdots \cdots,
\sigma_{\widetilde{k}}:=\lambda_1^2(V)\cdots \lambda_{\widetilde{k}}^2(V).
$$
If we define
$$
S_\alpha(V):=\sum_{j=1}^{\widetilde{k}}\left(\lambda_j^2(V)\right)^\alpha
$$
for each positive integer $\alpha\in\mathbb{N}^+$, then by the following well-known Newton's identities of elementary symmetric polynomials:
\begin{eqnarray*}
S_\alpha-S_{\alpha-1}\sigma_1+S_{\alpha-2}\sigma_2-\cdots+(-1)^{\alpha-1}S_1\sigma_{\alpha-1}+(-1)^\alpha \alpha\sigma_\alpha&=&0,\quad 1\leq \alpha\leq \widetilde{k},\\
S_\alpha-S_{\alpha-1}\sigma_1+S_{\alpha-2}\sigma_2-\cdots+(-1)^{\widetilde{k}}S_{\alpha-\widetilde{k}}\sigma_{\widetilde{k}}&=&0,\quad \alpha>\widetilde{k},
\end{eqnarray*}
 the functions $\widetilde{\phi}_1,\cdots,\widetilde{\phi}_\ell$ can  be expressed as polynomials in
$S_1,S_2,\cdots,S_\alpha$ for some $\alpha\in\mathbb{N}^+$ (see also Page 36 in \cite{Range}).
Thus $\widetilde{g}$ can actually be looked as a composite functions of $S_1,S_2,\cdots, S_\alpha$ for some $\alpha\in\mathbb{N}^+$.
Notice that
$$
S_\alpha(V)=\left(\lambda_1^2(V)\right)^\alpha+\cdots+\left(\lambda_{\widetilde{k}}^2(V)\right)^2=\mbox{tr}\left\{(VV^\ast)^\alpha\right\},
$$
which implies that $S_\alpha(V)=0$ if and only if $V=\pmb{0}$ for any $\alpha\in\mathbb{N}^+$.
Since $\pmb{f}$ is continuous on $T_0^{1,0}\mathfrak{D}$ and smooth on $\widetilde{T_0^{1,0}\mathfrak{D}}$ and satisfies $\pmb{f}(\lambda V)=\mid\lambda\mid\pmb{f}(V)$ for any $\lambda\in\mathbb{C}$, thus there exists some positive integer $k$ such that $\widetilde{g}$ is a composite function of
 \begin{eqnarray}
\xi_\alpha(V):&=&\sqrt[\alpha]{S_\alpha(V)}=\sqrt[\alpha]{\mbox{tr}\left\{(VV^\ast)^\alpha\right\}}\label{xia}
\end{eqnarray}
for $\alpha=1,\cdots,k$ since each $\xi_\alpha(V)$ is continuous on $T_{0}^{1,0}\mathfrak{D}$, smooth and positive on $\widetilde{T_0^{1,0}\mathfrak{D}}$ and satisfies $S_\alpha(\lambda V)=\mid\lambda\mid S_\alpha(V)$ for any $\lambda\in\mathbb{C}$.
So that there exists a positive integer $k (\geq \widetilde{k})$ and a continuous $1$-homogeneous function $g:[0,+\infty)^{k}\rightarrow [0,+\infty)$ which is smooth and positive on $(0,+\infty)^{k}$ such that $\tilde{g}=g(\xi_1(V),\cdots,\xi_{k}(V))$ with $\xi_\alpha(V)$ defined by \eqref{xia}. This completes the proof.
\end{proof}

The following proposition shows that there are lots of $\mbox{Iso}(\mathfrak{R}_A)$-invariant complex norms on $T_0^{1,0}\mathfrak{R}_A$ which are strongly pseudoconvex.

\begin{proposition}\label{prop-1}
Let $g(\xi_1,\cdots,\xi_{k})$ be the function in Theorem \ref{thm-1} such that
\begin{equation}\pmb{f}^2(V)=g\left(h_1(V),\cdots,h_{k}(V)\right)\label{f2}
\end{equation}
with
$$
h_\alpha(V):=\sqrt[\alpha]{\mbox{tr}\left\{(VV^\ast)^\alpha\right\}},\quad \forall V\in T_0^{1,0}\mathfrak{D}\quad\mbox{for}\quad\alpha=1,\cdots,k.
$$
If
\begin{equation}
\left(\frac{\partial^2 g}{\partial \xi_\alpha\partial \xi_\beta}\right)\geq 0\quad
\mbox{and}\quad \frac{\partial g}{\partial \xi_\alpha}>0,\quad \forall \alpha=1,\cdots,k.
\label{scc}
\end{equation}
Then $\pmb{f}$ is an $\mbox{Iso}(\mathfrak{D})$-invariant complex Minkowski norm.
\end{proposition}

\begin{proof}We only prove the case $\mathfrak{D}=\mathfrak{R}_I$, the proofs for $\mathfrak{D}=\mathfrak{R}_{II}$ and $\mathfrak{R}_{III}$ go essentially the same lines. Suppose that $g$ satisfies \eqref{scc}. It suffices to show that the complex norm $\pmb{f}$ defined by \eqref{f2} satisfies
\begin{equation}
\sum_{i,a=1}^m\sum_{j,b=1}^n\frac{\partial^2\pmb{f}^2}{\partial V_{ij}\partial\overline{V_{ab}}}W_{ij}\overline{W_{ab}}\geq 0\label{SC}
\end{equation}
for any $W=(W_{ij})\in T_{0}^{1,0}\mathfrak{R}_I \cong\mathscr{M}(m,n)$ with $W\neq \pmb{0}$ and equality holds if and only if $W=\pmb{0}$.

Indeed, we have
\begin{equation}
\frac{\partial^2\pmb{f}^2}{\partial V_{ij}\partial\overline{V_{ab}}}=\sum_{\alpha,\beta=1}^{k}\frac{\partial^2g}{\partial \xi_\alpha\partial\xi_\beta}\frac{\partial h_\alpha}{\partial V_{ij}}\frac{\partial h_\beta}{\partial \overline{V_{ab}}}+\sum_{\alpha=1}^{k}\frac{\partial g}{\partial \xi_\alpha}\frac{\partial^2h_\alpha}{\partial V_{ij}\partial\overline{V_{ab}}},\label{fij}
\end{equation}
and a direct computation gives
\begin{eqnarray}
\sum_{i=1}^m\sum_{j=1}^n\frac{\partial h_\alpha}{\partial V_{ij}}W_{ij}&=&\left\{\mbox{tr}\left[(VV^\ast)^\alpha\right]\right\}^{\frac{1}{\alpha}-1} \mbox{tr}\left[WV^\ast(VV^\ast)^{\alpha-1}\right],\label{v-a}\\
\sum_{a=1}^m\sum_{b=1}^n\frac{\partial h_\beta}{\partial \overline{V_{ab}}}\overline{W_{ab}}&=&\left\{\mbox{tr}\left[(VV^\ast)^\beta\right]\right\}^{\frac{1}{\beta}-1} \mbox{tr}\left[VW^\ast(VV^\ast)^{\beta-1}\right]\label{v-b}
\end{eqnarray}
and
\begin{eqnarray}
&&\sum_{i,a=1}^m\sum_{j,b=1}^n\frac{\partial^2h_{\alpha}}{\partial V_{ij}\partial\overline{V_{ab}}}W_{ij}\overline{W_{ab}}\nonumber\\
&=&\left\{\mbox{tr}\left[(VV^\ast)^\alpha\right]\right\}^{\frac{1}{\alpha}-1}\mbox{tr}\left[WW^\ast(VV^\ast)^{\alpha-1}\right]\nonumber\\
&&+\{\mbox{tr}[(VV^\ast)^\alpha]\}^{\frac{1}{\alpha}-2}\sum_{i=0}^{\alpha-2}\Bigg\{\mbox{tr}\left[(VV^\ast)^\alpha\right]
\mbox{tr}\left[W(V^\ast V)^{i+1}W^\ast(VV^\ast)^{\alpha-i-2}\right]\nonumber\\
&&-\mbox{tr}\left[VW^\ast(VV^\ast)^{\alpha-1}\right]\mbox{tr}\left[WV^\ast(VV^\ast)^{\alpha-1}\right]\Bigg\}.\label{v-c}
\end{eqnarray}

If $k=1$, we have $\pmb{f}^2(V)=c\cdot \mbox{tr}\{VV^\ast\}$ with $c:=g(1)>0$, $\pmb{f}$ is clear a complex Minkowski norm. If $k\geq 2$, then by (3.6) in Proposition 3.5 in \cite{GZ}, we have
\begin{eqnarray*}
&&\mbox{tr}[(VV^\ast)^\alpha]
\mbox{tr}[W(V^\ast V)^{i+1}W^\ast(VV^\ast)^{\alpha-i-2}]\\
&&-\mbox{tr}[VW^\ast(VV^\ast)^{\alpha-1}]\mbox{tr}[WV^\ast(VV^\ast)^{\alpha-1}]\geq 0
\end{eqnarray*}
for any fixed integer $\alpha\in\{1,\cdots,k\}$ and $i\in\{0,1,\cdots,k-2\}$.
Notice that $\mbox{tr}[(VV^\ast)^\alpha]\geq 0$ and $\mbox{tr}[WW^\ast(VV^\ast)^{\alpha-1}]\geq 0$, these together with the assumptions \eqref{scc} implies  the inequality \eqref{SC}.
This completes the proof.
\end{proof}

\begin{remark}\label{rem-a}
If $\pmb{f}$ is an $\mbox{Iso}(\mathfrak{D})$-invariant complex Minkowski norm which is also Hermitian quadratic, then it necessary that
$\frac{\partial g}{\partial \xi_\alpha}=0$ for $\alpha=2,\cdots,k$. In this case, we have $\pmb{f}^2(V)=c\cdot \mbox{tr}\{VV^\ast\}$ for any $V\in T_0^{1,0}\mathfrak{D}$ and $c:=g(1)$ is a positive constant.
If $\pmb{f}$ is an $\mbox{Iso}(\mathfrak{D})$-invariant complex Minkowski norm which is non-Hermitian quadratic, we must have
$\pmb{f}^2(V)=\xi_1(V)\psi(V)$ for any $V\in \widetilde{T_0^{1,0}\mathfrak{D}}$ with
\begin{equation}
\psi(V):=g\left(1,\frac{\xi_2(V)}{\xi_1(V)},\cdots,\frac{\xi_{k}(V)}{\xi_1(V)}\right).
\end{equation}
 It is clear that $\psi(\lambda V)=\psi(V)$ for any $\lambda\in\mathbb{C}^\ast$ and $V\in \widetilde{T_0^{1,0}\mathfrak{D}}$, namely $\psi(V)$ is actually a smooth function defined on $PT_0^{1,0}\mathfrak{D}\cong\mathbb{CP}^{r-1}$ with $r=\dim\,\mathfrak{D}$. Thus there exists two positive constant
$$c_1=\min_{V\in PT_0^{1,0}\mathfrak{D}}\psi(V)\quad\mbox{and}\quad c_2=\max_{V\in PT_0^{1,0}\mathfrak{D}}\psi(V)$$
such that
\begin{equation}
c_1\cdot \mbox{tr}\{VV^\ast\}\leq \pmb{f}^2(V)\leq c_2\cdot \mbox{tr}\{VV^\ast\},\quad \forall V\in T_0^{1,0}\mathfrak{D}.\label{c1c2}
\end{equation}
\end{remark}

For simplicity, in the following if $\pmb{f}:T_0^{1,0}\mathfrak{D}\rightarrow [0,+\infty)$ is the associated complex Minkowski norm of a strongly pseudoconvex complex Finsler metric $F:T^{1,0}\mathfrak{D}\rightarrow[0,+\infty)$, we always assume that the corresponding function $g$  of $\pmb{f}$ satisfies the condition \eqref{scc}.

The following theorem gives a complete characterizations of $\mbox{Aut}(\mathfrak{D})$-invariant strongly pseudoconvex complex Finsler metrics on $\mathfrak{D}=\mathfrak{R}_I,\mathfrak{R}_{II},\mathfrak{R}_{III}$.

\begin{theorem}\label{thm-2}
Let $F:T^{1,0}\mathfrak{D}\rightarrow[0,+\infty)$ be an $\mbox{Aut}(\mathfrak{D})$-invariant strongly pseudoconvex complex Finsler metric on a classical domain $\mathfrak{D}$ ($\mathfrak{D}=\mathfrak{R}_I,\mathfrak{R}_{II},\mathfrak{R}_{III}$).
Then we have the following assertions:

(1) There exists a positive integer $k\in\mathbb{N}^+$ and a continuous $1$-homogeneous function $g:[0,+\infty)^{k}\rightarrow [0,+\infty)$ which is smooth and positive on $(0,+\infty)^k$ such that
\begin{equation}
F^2(Z;V)=g\left(\widetilde{h}_1(Z;V),\cdots, \widetilde{h}_{k}(Z;V)\right),\quad \forall (Z;V)\in T^{1,0}\mathfrak{D}\label{FIb}
\end{equation}
with
\begin{equation}
\widetilde{h}_\alpha(Z;V):=\sqrt[\alpha]{\mbox{tr}\left\{\left[(I-ZZ^\ast)^{-1}V(I-Z^\ast Z)^{-1}V^\ast\right]^\alpha\right\}},\quad \alpha=1,\cdots,k.
\label{tha}
\end{equation}

In particular, if $F$ is  Hermitian quadratic,  then
\begin{equation}
F^2(Z;V)=c\cdot\mbox{tr}\left[(I-ZZ^\ast)^{-1}V(I-Z^\ast Z)^{-1}V^\ast\right], \quad \forall (Z;V)\in T^{1,0}\mathfrak{D}\label{itha}
\end{equation}
namely  $F$ is a constant multiple of the Bergman metric on $\mathfrak{D}$.

(2) $F$ is a K\"ahler-Finsler metric. More precisely $F$ is a  complete K\"ahler-Berwald metric on $\mathfrak{D}$.

(3) The holomorphic sectional curvature $K_{F}$ of $F$ is bounded below and above by negative constants, respectively. That is, there exists constants $-K_1$ and $-K_2$ such that
\begin{equation}
-K_1\leq K_{F}(Z;V)\leq -K_2<0,\quad \forall (Z; V)\in \widetilde{T^{1,0}\mathfrak{D}}.\label{HSC-I}
\end{equation}
The holomorphic bisectional curvature $B_{F}$ of $F$ is  non-positive and bounded below. That is, there exists a constant $-C$ such that
\begin{equation}
-C\leq B_{F}(Z;V,W)\leq 0,\quad \forall(Z;V),(Z;W)\in\widetilde{T^{1,0}\mathfrak{D}}.\label{BHSC-I}
\end{equation}
\end{theorem}

\begin{proof}  (1) Since $\mathfrak{R}_I$ is a homogeneous domain, it suffices to determine the expression of $F(Z_0;V)$ for any fixed point $Z_0\in\mathfrak{R}_I$ and any tangent vector $V\in T_{Z_0}^{1,0}\mathfrak{R}_I$.

By Theorem 2.3 in \cite{GZ}, there exists an $\mbox{Iso}(\mathfrak{R}_I)$-invariant complex Minkowski norm $\pmb{f}:T_0^{1,0}\mathfrak{R}_I\rightarrow [0,+\infty)$ such that
\begin{equation}
F(Z_0;V)=\pmb{f}((\varPhi_{Z_0})_\ast(V)),\label{FZ0V}
\end{equation}
where $\varPhi_{Z_0}\in\mbox{Aut}(\mathfrak{R}_I)$ satisfies $\varPhi_{Z_0}(Z_0)=0$ and $(\varPhi_{Z_0})_\ast$ denotes the differential of $\varPhi$ at the point $Z_0\in\mathfrak{R}_I$.
Note that for any chosen $\pmb{f}$, \eqref{FZ0V} is independent of the choice of $\varPhi_{Z_0}\in\mbox{Aut}(\mathfrak{R}_I)$ with $\varPhi_{Z_0}(Z_0)=0$. Thus by \eqref{p-I}, we have
\begin{equation}
(\varPhi_{Z_0})_\ast(V)=AV(I_n-Z_0^\ast Z_0)^{-1}D^{-1},\quad \forall V\in T_{Z_0}^{1,0}\mathfrak{R}_I,\label{PVI}
\end{equation}
where $A$ and $D$ satisfy \eqref{p-Ia}.

By Theorem \ref{thm-1}, there exists a positive integer $k\in\mathbb{N}^+$ and a  continuous $1$-homogeneous function $g:[0,+\infty)^{k}\rightarrow [0,+\infty)$ which is smooth and positive on $(0,+\infty)^{k}$ such that
 \begin{equation}
\pmb{f}^2(V)=g\left(h_1(V),\cdots,h_{k}(V)\right)\label{f2V}
\end{equation}
with
\begin{equation}
h_\alpha(V):=\sqrt[\alpha]{\mbox{tr}\left\{(VV^\ast)^\alpha\right\}},\quad\alpha=1,\cdots,k.\label{hlV}
\end{equation}
Setting $\widetilde{h}_\alpha(Z_0;V):=h_\alpha((\varPhi_{Z_0})_\ast(V))$ for $\alpha=1,\cdots,k$,  and using \eqref{FZ0V}-\eqref{f2V},  we obtain
\begin{equation}
F^2(Z_0;V)=g\left(\widetilde{h}_1(Z_0;V),\cdots, \widetilde{h}_{k}(Z_0;V)\right)\label{FZ0}
\end{equation}
with
$$
\widetilde{h}_\alpha(Z_0;V)=\sqrt[\alpha]{\mbox{tr}\left\{\left[I_m-Z_0Z_0^\ast)^{-1}V(I_n-Z_0^\ast Z_0)^{-1}V^\ast\right]^\alpha\right\}}
$$
 for $\alpha=1,\cdots,k$. Since \eqref{FZ0} holds for any fixed $Z_0\in\mathfrak{R}_I$ and any tangent vector $V\in T_{Z_0}^{1,0}\mathfrak{R}_I$,
changing $Z_0$ into $Z$ if necessary, we obtain \eqref{FIb} and \eqref{tha}.

If $k=1$, then \eqref{FIb} reduces to
\begin{equation}
F^2(Z;V)=c\cdot \mbox{tr}\left\{(I_m-ZZ^\ast)^{-1}V(I_n-Z^\ast Z)^{-1}V^\ast\right\}
\end{equation}
with $c:=g(1)>0$, namely $F$ is a constant multiple of the Bergman metric on $\mathfrak{R}_I$.

If $k\geq 2$ and $\frac{\partial g}{\partial \xi_\alpha}\neq 0$ for some $\alpha\in\{2,\cdots,k\}$, then $F^2(Z;V)$ is non-Hermitian quadratic with respect to $V\in T_{Z}^{1,0}\mathfrak{R}_I$.

(2) By Theorem 2.4 in  \cite{GZ}, in order to show that every $\mbox{Aut}(\mathfrak{R}_I)$-invariant strongly pseudoconvex complex Finsler metric $F$ is necessary a K\"ahler-Berwald metric, it suffices to show that at the point $(0;V)\in T^{1,0}\mathfrak{R}_I$ with $V\neq \pmb{0}$, we have
\begin{equation}
\frac{\partial^2F^2}{\partial Z_{ij}\partial \overline{V_{ab}}}(0;V)=0,\quad \forall i,a=1,\cdots,m; j,b=1,\cdots,n.\label{F0V}
\end{equation}

Indeed, since we have
\begin{eqnarray*}
\frac{\partial^2F^2}{\partial Z_{ij}\partial \overline{V_{ab}}}
&=&\sum_{\alpha,\beta=1}^{k}\frac{\partial^2g}{\partial \xi_\alpha\partial \xi_\beta}\frac{\partial \widetilde{h}_\alpha}{\partial Z_{ij}}\frac{\partial \widetilde{h}_\beta}{\partial \overline{V_{ab}}}
+\sum_{\alpha,\beta=1}^{k}\frac{\partial g}{\partial \xi_\alpha}\frac{\partial^2 \widetilde{h}_\alpha}{\partial Z_{ij}\partial\overline{V_{ab}}}.
\end{eqnarray*}
At the point $(0;V)$ with $V\neq \pmb{0}$, it is easy to check that
\begin{eqnarray*}
\frac{\partial \widetilde{h}_\alpha}{\partial Z_{ij}}(0;V)&=&0,\quad
\frac{\partial^2 \widetilde{h}_\alpha}{\partial Z_{ij}\partial \overline{V_{ab}}}(0;V)=0,\quad \forall i,a=1,\cdots,m; j,b=1,\cdots,n,
\end{eqnarray*}
 which implies \eqref{F0V}.  Furthermore, it is easy to check that the horizontal Chern-Finsler connection coefficients $\varGamma_{j;k}^l$ associated to $F$ coincide with the Hermitian connection coefficients associated to the Bergman metric on $\mathfrak{R}_I$. Therefore $F$ enjoys the same system of equations of geodesic as that of the Bergman metric on $\mathfrak{R}_I$. Hence the geodesics (as a set) are the same as that of the Bergman metric on $\mathfrak{R}_I$, which implies that $F$ is complete.

(3) By a direct calculation, at the point $(0;V)$ with $V\neq \pmb{0}$, we have $\frac{\partial \widetilde{h}_\alpha}{\partial Z_{ij}}(0;V)=0$ and
\begin{eqnarray*}
\frac{\partial^2\widetilde{h}_\alpha}{\partial Z_{ij}\partial \overline{Z_{ab}}}(0;V)=\left\{\mbox{tr}[(VV^\ast)^\alpha]\right\}^{\frac{1}{\alpha}-1}\mbox{tr}\left\{\left(E_{ij}E_{ab}'VV^\ast+VE_{ab}'E_{ij}V^\ast\right)(VV^\ast)^{\alpha-1}\right\},
\end{eqnarray*}
from which it follows that
\begin{eqnarray*}
\sum_{i,a=1}^m\sum_{j,b=1}^n\frac{\partial^2 \widetilde{h}_\alpha}{\partial Z_{ij}\partial\overline{Z_{ab}}}V_{ij}\overline{V_{ab}}&=&2\left\{\mbox{tr}[(VV^\ast)^\alpha]\right\}^{\frac{1}{\alpha}-1}\cdot\mbox{tr}\left\{(VV^\ast)^{\alpha+1}\right\}> 0
\end{eqnarray*}
for any $V\in\widetilde{T_{0}^{1,0}\mathfrak{R}_I}$. Thus by Proposition 2.6 in \cite{GZ}, the holomorphic sectional curvature $K_{F}(0;V)$ of $F$ along $V\in\widetilde{T_{0}^{1,0}\mathfrak{R}_I}$ is given by
\begin{eqnarray}
K_{F}(0;V)=-\frac{2}{F^4(0;V)}\sum_{\alpha=1}^{k}\frac{\partial g}{\partial\xi_\alpha}\frac{\partial^2\widetilde{h}_\alpha}{\partial Z_{ij}\partial\overline{Z_{ab}}}V_{ij}\overline{V_{ab}}<0\label{hsc-a}
\end{eqnarray}
since $\frac{\partial g}{\partial \xi_\alpha}>0$ for $\alpha=1,\cdots,k$. Note that $K_{F}(0;\lambda V)=K_{F}(0;V)$ for any $\lambda\in\mathbb{C}^\ast$ and $V\in\widetilde{T_{0}^{1,0}\mathfrak{R}_I}$. This shows that $K_{F}(0;V)$ actually defines a smooth function on the projectivized fiber $PT_{0}^{1,0}\mathfrak{R}_I\cong \mathbb{CP}^{mn-1}$ which is compact. So that there exist two  negative constants $-K_1$ and $-K_2$ such that
\begin{equation}
-K_1\leq K_{F}(0;V)\leq -K_2<0,\quad\forall V\in \widetilde{T_{0}^{1,0}\mathfrak{R}_I}.\label{hsc-lu}
\end{equation}
Since $\mathfrak{R}_I$ is a homogeneous domain and $K_{F}(\varPhi(Z);\varPhi_\ast(V))=K_{F}(Z;V)$ for any $\varPhi\in\mbox{Aut}(\mathfrak{R}_I)$ and $(Z;V)\in\widetilde{T^{1,0}\mathfrak{R}_I}$, \eqref{hsc-lu} actually holds for any $(Z;V)\in \widetilde{T^{1,0}\mathfrak{R}_I}$. Namely we have \eqref{HSC-I}.

Similarly, for any $V,W\in \widetilde{T_{0}^{1,0}\mathfrak{R}_I}$, we have
$$
\sum_{i,a=1}^m\sum_{j,b=1}^n\frac{\partial^2 \widetilde{h}_\alpha}{\partial Z_{ij}\partial\overline{Z_{ab}}}W_{ij}\overline{W_{ab}}=2\left\{\mbox{tr}[(VV^\ast)^\alpha]\right\}^{\frac{1}{\alpha}-1}\cdot\mbox{tr}\left\{WW^\ast(VV^\ast)^{\alpha}\right\}\geq0.
$$
Thus by Proposition 2.7 in \cite{GZ}, the holomorphic bisectional curvature $B_F(0;V,W)$ of $F$ along $V,W\in \widetilde{T_{0}^{1,0}\mathfrak{R}_I}$ is given by
\begin{eqnarray}
B_{F}(0;V,W)=-\frac{2}{F^2(0;V)F_I^2(0;W)}\sum_{\alpha=1}^{k}\frac{\partial g}{\partial\xi_\alpha}\frac{\partial^2\widetilde{h}_\alpha}{\partial Z_{ij}\partial\overline{Z_{ab}}}W_{ij}\overline{W_{ab}}\leq 0.\label{bhsc-b}
\end{eqnarray}
By the same reason as $K_{F}(0;V)$, it follows from \eqref{bhsc-b} that there exists a negative constant $-C$ such that $-C\leq B_{F}(0;V,W)\leq 0$ holds for any $V,W\in \widetilde{T_0^{1,0}\mathfrak{R}_I}$, which implies \eqref{BHSC-I} since $\mathfrak{R}_I$ is homogeneous.

Since $\mathfrak{R}_{II}(m)$ and $\mathfrak{R}_{III}(m)$ are complex submanifolds of $\mathfrak{R}_I(m,n)$, the proof for $\mathfrak{R}_{II}$ and $\mathfrak{R}_{III}$ goes essentially the same lines as that of  $\mathfrak{R}_I$. Indeed, using \eqref{D-II} and \eqref{D-III}, we can obtain the formulas corresponding to \eqref{FIb} and \eqref{tha}, respectively for $\mathfrak{R}_{II}$ and $\mathfrak{R}_{III}$. These formulas essentially share the same form as that of $\mathfrak{R}_I$.
\end{proof}

\begin{remark}
We point it out that \eqref{tha} and \eqref{itha} should be understood in the corresponding cases: for $Z\in \mathfrak{R}_{II}$, one has  $Z'=Z$ and a tangent vector $V\in T_Z^{1,0}\mathfrak{R}_{II}$ satisfies $V=V'$, hence we have $Z^\ast=\overline{Z}$ and $V^\ast=\overline{V}$;  for $Z\in\mathfrak{R}_{III}$, one has  $Z'=-Z$ and a tangent vector $V\in T_Z^{1,0}\mathfrak{R}_{III}$ satisfies $V'=-V$, hence we have $Z^\ast=-\overline{Z}$ and $V^\ast=-\overline{V}$.
\end{remark}

\subsection{The case of $\mathfrak{R}_{IV}(N)$}

The automorphism of $\mathfrak{R}_{IV}$ which transforms a fixed point $z_0\in\mathfrak{R}_{IV}$ into the origin is given in the form  (cf. Hua \cite{Hua} and Lu \cite{Lu})
	\begin{eqnarray}
		w&=&\phi_{z_0}(z)=\left\{\left[\left(\frac{1+zz'}{2},\frac{1-zz'}{2i}\right)-zX_0'\right]A\begin{pmatrix}
			1 \\
			i \\
		\end{pmatrix}
		\right\}^{-1}\nonumber\\
&&\times\left[z-\left(\frac{1+zz'}{2},\frac{1-zz'}{2i}\right)X_0\right]D,
		\label{p-IV}
	\end{eqnarray}
	where
	\begin{eqnarray}
		X_0&=&2\begin{pmatrix}
			z_0z_0'+1 & i\left(z_0z_0'-1\right) \\
			\overline{z_0}z_0^\ast+1 & -i\left(\overline{z_0}z_0^\ast-1\right) \\
		\end{pmatrix}^{-1}
		\begin{pmatrix}
			z_0 \\
			\overline{z_0} \\
		\end{pmatrix}\nonumber\\
		&=&\frac{-1}{1-\mid z_0z_0'\mid^2}\begin{pmatrix}
			\left(\overline{z_0}z_0^\ast-1\right)z_0+\left(z_0z_0'-1\right)\overline{z_0}  \\
			i\left(z_0z_0'+1\right)\overline{z_0}-i\left(\overline{z_0}z_0^\ast+1\right)z_0 \\
		\end{pmatrix}
		\label{p-IV-a}
	\end{eqnarray}
	is a real $2$-by-$N$ matrix satisfying $I-X_0X_0'>0$,
	$A\in\mathscr{M}(2)$ and $D\in\mathscr{M}(N)$ are real matrices satisfying
	\begin{equation}
		AA'=(I-X_0X_0')^{-1},\quad DD'=(I-X_0'X_0)^{-1},\quad \det A>0.\label{p-IV-b}
	\end{equation}
	
	Conversely, every element of $\mbox{Aut}(\mathfrak{R}_{IV})$ can be expressed into the form \eqref{p-IV} for some $z_0\in \mathfrak{R}_{IV}$ and $A\in\mathscr{M}(2),D\in\mathscr{M}(N)$ satisfying \eqref{p-IV-b}.

	By \eqref{p-IV}-\eqref{p-IV-b}, the isotropy  subgroup $\mbox{Iso}(\mathfrak{R}_{IV})$ at the origin consists of complex-linear mappings of the following form:
	\begin{equation}
		w=\phi_0(z)=e^{i\theta}zD,\quad  z\in\mathfrak{R}_{IV},
	\end{equation}
	where $\theta\in\mathbb{R}$ and $D\in O(N;\mathbb{R})$ is the orthogonal group of order $N$ over the real number field $\mathbb{R}$. Thus $(\phi_0)_\ast (\xi)=e^{i\theta}\xi D$ for any $\xi\in T_0^{1,0}\mathfrak{R}_{IV}\cong\mathbb{C}^N$.
Thus a complex norm $\pmb{f}:T_0^{1,0}\mathfrak{R}_{IV}\rightarrow[0,+\infty)$ is $\mbox{Iso}(\mathfrak{R}_{IV})$-invariant if and only if
\begin{eqnarray}
	\pmb{f}(\xi)&=&\pmb{f}(e^{i\theta}\xi D)\label{gA}
\end{eqnarray}
for any $\theta\in \mathbb{R}$ and $D\in O(N;\mathbb{R})$.

\begin{theorem}\label{f4c}
Let $\pmb{f}: T_0^{1,0}\mathfrak{R}_{IV}\rightarrow [0,+\infty)$ be an $\mbox{Iso}(\mathfrak{R}_{IV})$-invariant complex norm  on $T_0^{1,0}\mathfrak{R}_{IV}$. Then

(1) there exists a continuous  function $\phi:[0,1]\rightarrow (0,+\infty)$ such that
\begin{eqnarray}
	\pmb{f}(\xi)=\sqrt{r\phi(s)},\quad r=\xi\xi^\ast, \quad s=\frac{\mid\xi\xi'\mid^2}{r^2},\quad \pmb{0}\neq \xi\in  T_{0}^{1,0}\mathfrak{R}_{IV}\cong\mathbb{C}^N;\label{g}
\end{eqnarray}

(2) if $\pmb{f}$ is smooth on $\widetilde{T_0^{1,0}\mathfrak{R}_{IV}}$, then $\phi$ is also smooth on $[0,1]$, and $\pmb{f}$ is a complex Minkowski norm  if and only if
		\begin{equation}
			\phi-2s\phi'>0\quad \mbox{and}\quad \phi[\phi+2(2-3s)\phi']+4s(1-s)[\phi \phi''-(\phi')^2]>0,\label{sn}
		\end{equation}
where $\phi'$ and $\phi''$ are the derivatives of $\phi$ with respect to $s$, respectively.
\end{theorem}
\begin{proof} In view of  Proposition 3.29 in \cite{GZ}, we only need to prove assertion (1).

For any non-zero vector $\xi=(\xi_1,\cdots,\xi_N)\in \mathbb{C}^N,$  we  denote  $\xi=\mbox{Re}\,\xi+i\mbox{Im}\,\xi$ with $\mbox{Re}\,\xi=(\mbox{Re}\,\xi_1,\cdots,\mbox{Re}\,\xi_N)$ and $\mbox{Im}\,\xi=(\mbox{Im}\,\xi_1,\cdots,\mbox{Im}\,\xi_N)\in\mathbb{R}^N$.

(i) If $\mbox{Re}\,\xi\neq \pmb{0}$, we can take $x:= \frac{\mbox{Re}\,\xi}{\sqrt{\langle\mbox{Re}\,\xi,\mbox{Re}\,\xi\rangle}}\in\mathbb{R}^N$ so that $\|x\|^2=1$, here $\|\cdot\|$ denotes the usual Euclidean norm.  Take another unit vector $y\in\mathbb{R}^N$ such that $y$ is orthogonal to $x$, namely $\|y\|^2=1$ and $\langle x,y\rangle=0$. Expressing $\xi$ in terms of  $x$ and $y$, we get
$\xi=(R+iS)x+iTy$,
where
\begin{eqnarray*}
R+iS&=&\frac{\xi\xi'+\xi\xi^\ast}{\sqrt{\xi\xi'+\overline{\xi}\xi^\ast+2\xi\xi^\ast}},\quad
T=\frac{\sqrt{(\xi\xi^\ast)^2-\mid\xi\xi'\mid^2}}{\sqrt{\xi\xi'+\overline{\xi}\xi^\ast+2\xi\xi^\ast}}.
\end{eqnarray*}

Now taking an orthogonal matrix $A\in O(N;\mathbb{R})$ (possibly depends on $\xi$) with $x'$ and $y'$ being the first and second columns of $A$, respectively. Here $x'$ and $y'$ denote the transpose of the row vector $x$ and $y$, respectively. Then we have
\begin{equation} \xi A=\big(R+iS,iT,\underbrace{0,\cdots,0}_{N-2}\big).
\label{v-rsting}
\end{equation}
Since $\pmb{f}$ satisfies \eqref{gA},  for any $\xi\in\mathbb{C}^N$ with $\xi\neq \pmb{0}$, we have
\begin{eqnarray}
\pmb{f}(\xi)
&=&\pmb{f}\left(\left(R+iS,iT,\underbrace{0,\cdots,0}_{N-2}\right)\right)\nonumber\\
&=&\sqrt{\xi\xi^\ast}\pmb{f}\left(\frac{\zeta+1}{\sqrt{2+\zeta+\bar{\zeta}}},i\frac{\sqrt{1-\mid\zeta\mid^2}}{\sqrt{2+\zeta+\bar{\zeta}}},\underbrace{0,\cdots,0}_{N-2}\right),\label{fx}
\end{eqnarray}
where we denote $\zeta:=\frac{\xi\xi'}{\xi\xi^\ast}$. Setting $\xi\xi':=\rho e^{i\theta}$ for $\rho=\mid\xi\xi'\mid$ and $\theta\in\mathbb{R}$. Replacing $\xi$ with $e^{-\frac{i}{2}\theta}\xi$ in \eqref{fx} and using \eqref{gA}, we obtain
\begin{equation}
\pmb{f}(\xi)=\frac{\sqrt{\xi\xi^\ast}}{\sqrt{2}}\pmb{f}\left(\sqrt{1+\mid\zeta\mid},i\sqrt{1-\mid\zeta\mid},\underbrace{0,\cdots,0}_{N-2}\right).
\end{equation}

Since $\pmb{f}$ is  continuous on $\mathbb{C}^N$ and smooth on $\mathbb{C}^N\setminus\{\pmb{0}\}$, if we denote $s:=\mid\zeta\mid^2=\frac{\mid\xi\xi'\mid^2}{r^2}$ with $r=\xi\xi^\ast$ for any $\xi\in \mathbb{C}^N$ satisfying $\xi\neq \pmb{0}$, then
we can rewrite $\pmb{f}$ as $\pmb{f}(\xi)=\sqrt{r\phi(s)}$ with
$$
\phi(s)=\frac{1}{2}\pmb{f}\left(\sqrt{1+\sqrt{s}},i\sqrt{1-\sqrt{s}},\underbrace{0,\cdots,0}_{N-2}\right).
$$

Note that we always have $0\leq s\leq 1$, and since $\mbox{Re}\,\xi\neq \pmb{0}$, $s=1$ if and only if $\mbox{Im}\,\xi=\pmb{0}$.

(ii) If $\mbox{Re}\,\xi=\pmb{0}$, we have $R=0$ and $x=\pmb{0}$. Since $\xi\neq \pmb{0}$,  we must have $\mbox{Im}\,\xi\neq \pmb{0}$. Thus we can take a unit real vector $y=\frac{\mbox{Im}\,\xi}{\sqrt{\langle \mbox{Im}\,\xi,\mbox{Im}\,\xi\rangle}}$ such that $\|y\|^2=1$ and $\xi=i\sqrt{\xi\xi^\ast}y\neq \pmb{0}$.
Now taking an orthogonal matrix $A\in O(N;\mathbb{R})$ (possibly depends on $\xi$) with $y'$ (the transpose of the row vector $y$) being the first column of $A$, then we have
\begin{equation}
\xi A=\left(i\sqrt{\xi\xi'},\underbrace{0,\cdots,0}_{N-1}\right),
\end{equation}
Hence
$$
\pmb{f}(\xi)=\pmb{f}\left(i\sqrt{\xi\xi^\ast},\underbrace{0,\cdots,0}_{N-1}\right)
=\sqrt{\xi\xi^\ast}\pmb{f}\left(i,\underbrace{0,\cdots,0}_{N-1}\right)=c\sqrt{\xi\xi^\ast},
$$
with $c:=\pmb{f}\left(i,\underbrace{0,\cdots,0}_{N-1}\right)>0$.
This completes the proof.
\end{proof}

\begin{remark}
 The complex norm $\pmb{f}: T_0^{1,0}\mathfrak{R}_{IV}\rightarrow [0,+\infty)$ of the form \eqref{g} was first introduced in \cite{GZ} which  was used to give an explicit construction of $\mbox{Aut}(\mathfrak{R}_{IV})$-invariant strongly pseudoconvex complex Finsler metrics on $\mathfrak{R}_{IV}$. The assertion (1) in Theorem \ref{f4c} actually shows that any $\mbox{Iso}(\mathfrak{R}_{IV})$-invariant complex norm is necessary of the form \eqref{g}.
\end{remark}

\begin{theorem}\label{thm-7}
Let $F: T^{1,0}\mathfrak{R}_{IV}\rightarrow [0,+\infty)$ be an $\mbox{Aut}(\mathfrak{R}_{IV})$-invariant strongly pseudoconvex complex Finsler metric on $\mathfrak{R}_{IV}$.
Then we have the following assertions:

(1) There exists a smooth function $\phi:[0,1]\rightarrow[0,+\infty)$
 such that
\begin{equation}
F(z;v)=\sqrt{\widetilde{r}(z;v)\phi(\widetilde{s}(z;v))},\quad \forall (z;v)\in \widetilde{T^{1,0}\mathfrak{R}_{IV}}
\end{equation}
 with
\begin{equation}
\widetilde{r}(z;v):=\frac{2N}{\Delta^2(z)}v\left[\triangle(z)I_N-2(\overline{z}z^\ast)z'z-2(1-2zz^\ast)z'\overline{z}
			+2z^\ast z-2(zz')z^\ast\overline{z}
			\right]v^\ast\label{rzv}
\end{equation}
and
\begin{equation}
 \widetilde{s}(z;v):=\frac{4N^2}{\triangle^2(z)}\frac{\mid vv'\mid^2}{\widetilde{r}^2(z;v)},\quad \triangle (z):=1+\mid zz'\mid-2zz^\ast.\label{szv}
\end{equation}
In particular, if $F$ is Hermitian quadratic, then (up to multiplicative constants) it is the Bergman metric on $\mathfrak{R}_{IV}$.

(2) $F$ is a K\"ahler-Finsler metric. More precisely, $F$ is a  complete K\"ahler-Berwald metric on $\mathfrak{R}_{IV}$.

(3) The holomorphic sectional curvature $K_F$ of $F$ is bounded below and above by negative constants, respectively. That is, there exists constants $-K_1$ and $-K_2$ such that
$$-K_1\leq K_F(z;v)\leq -K_2<0,\quad \forall (z;v)\in \widetilde{T^{1,0}\mathfrak{R}_{IV}}.$$
If $\phi'\geq 0$, then the holomorphic bisectional curvature $B_F$ of $F$ is non-positive and bounded below. That is, there exists a constant $-C$ such that
$$
-C\leq B_F(z;v,w)\leq 0,\quad \forall (z;v),(z;w)\in\widetilde{T^{1,0}\mathfrak{R}_{IV}}.
$$
\end{theorem}

\begin{proof}
Let $z_0\in\mathfrak{R}_{IV}$ be an arbitrary fixed point and $v\in T_0^{1,0}\mathfrak{R}_{IV}$. Let $\psi_{z_0}\in\mbox{Aut}(\mathfrak{R}_{IV})$ be given by \eqref{p-IV} such that $\psi_{z_0}(z_0)=0$ and $(\psi_{z_0})_\ast(v)\in T_0^{1,0}\mathfrak{R}_{IV}$.
By the assertion (1) in Theorem  \ref{f4c}, there exists a smooth function $\phi:[0,1]\rightarrow (0,+\infty)$ such that
 $\pmb{f}^2(v)=r\phi(s)$ with $s:=\frac{\mid vv'\mid^2}{r^2}$ and $r:=vv^\ast$
for any $v\in\mathbb{C}^N\setminus\{\pmb{0}\}$. Define $F(z_0;v):=\pmb{f}(\sqrt{2N}(\psi_{z_0})_\ast(v))$, it is easy to check (cf. Example 3.32 in \cite{GZ}) that
\begin{equation}
F^2(z_0;v)=\sqrt{\widetilde{r}(z_0;v)\phi(\widetilde{s}(z_0;v)},\label{FZ0-IV}
\end{equation}
 where
\begin{eqnarray*}
\widetilde{r}(z_0;v)&=&\frac{2N}{\Delta^2(z_0)}v\left[\triangle(z_0)I_N-2(\overline{z_0}z_0^\ast)z_0'z_0\right.\\
&&\left.-2(1-2z_0z_0^\ast)z_0'\overline{z_0}
			+2z_0^\ast z_0-2(z_0z_0')z_0^\ast\overline{z_0}
			\right]\overline{v'}
\end{eqnarray*}
and
$$\widetilde{s}(z;v)=\frac{4N^2}{\triangle^2(z)}\frac{\mid vv'\mid^2}{\widetilde{r}^2(z;v)},\quad\quad \triangle(z_0):=1+\mid z_0z_0'\mid^2-2z_0z_0^\ast.$$
Since \eqref{FZ0-IV} holds for any fixed point $z_0\in\mathfrak{R}_{IV}$ and any $v\in \widetilde{T_{z_0}^{1,0}\mathfrak{R}_{IV}}$, changing $z_0$ to $z$ if necessary
we obtain \eqref{rzv} and \eqref{szv}.

 The assertions (2) and (3) follows immediately from  Theorem 3.35 and Theorem 3.36 in \cite{GZ}.
\end{proof}

\section{Schwarz lemma on the classical domains}\label{section-4}

In this section,  we shall prove Theorem \ref{m-th-2}. Our proof is purely geometric. First we need establish a comparison theorem of holomorphically invariant K\"ahler-Berwald metrics with the Carath\'edory metric on convex domains. Indeed, by Lempert's fundamental theorem of invariant metrics on convex domains the Carath\'eodory metric and the Kobayashi metric coincide (cf. \cite{Lempert}, see also \cite{AP}). In particular, this will be applied to the classical domains for our purpose.

\subsection{Holomorphic curvature of complex Finsler metric}

In complex Finsler geometry, the notion of holomorphic curvature is defined both for  upper semi-continuous complex Finsler metrics  and strongly pseudoconvex complex Finsler metrics.

For an upper semi-continuous complex Finsler metric $F:T^{1,0}M\rightarrow[0,+\infty)$, the holomorphic curvature $K_F(p;v)$ at $p\in M$ along a non-zero vector $v\in T_p^{1,0}M$ is given by (cf. Definition 3.1.8 in \cite{AP})
$$K_F(p;v):=\sup_\varphi\left\{K(\varphi^\ast G)(0)\right\},$$
where the supremum is taken with respect to the family of all holomorphic mappings $\varphi$ from the open unit disk $\mathbb{D}\subset\mathbb{C}$ into $M$ with $\varphi(0)=p$ and $\varphi'(0)=\lambda v$ for some $\lambda\in\mathbb{C}^\ast$, and $K(\varphi^\ast G)$ is the Gaussian curvature (with respect which the Gaussian curvature of the Poincar\'e metric $ds_{\mathbb{D}}^2=\frac{dz\otimes d\bar{z}}{(1-\mid z\mid^2)^2}$ is $-4$) of the pseudohermitian metric $\varphi^\ast G$ on $\mathbb{D}$ (cf. Definition 2.5.3 in \cite{AP}). For a strongly pseudoconvex complex Finsler metric $F:T^{1,0}M\rightarrow[0,+\infty)$, the holomorphic sectional curvature $K_F$ can also be defined via the curvature of the Chern-Finsler connection associated to $F$ (cf. Definition 2.5.2 in \cite{AP}). By Corollary 2.5.4 in \cite{AP}, if $F$ is smooth these two definitions coincide.

Let $\mbox{Aut}(M)$ be the holomorphic automorphism group of a complex manifold $M$. For any $p\in M$ and a non-zero vector $v\in T_p^{1,0}M$, one has
$$
K_F(p;v)=K_F(\phi(p);\phi_\ast(v)),\quad \forall \phi\in\mbox{Aut}(M)
$$
and
$$
K_F(p;\lambda v)=K_F(p;v),\quad \forall\lambda\in\mathbb{C}^\ast.
$$
Thus $K_F$ is $\mbox{Aut}(M)$-invariant and it is actually a smooth function defined on the projective bundle $PT^{1,0}M$ over $M$.
Furthermore if $M$ is a homogeneous complex manifold and $F$ is an $\mbox{Aut}(M)$-invariant strongly pseudoconvex complex Finsler metric on $M$, then there are points $(z_1;[v_1])$ and $(z_2,[v_2])\in PT^{1,0}M$ such that the minimum  and maximum of $K_F$ can be achieved at $(z_1;[v_1])$ and $(z_2;[v_2])$, respectively.

In the following Proposition \ref{pmp}, we  establish a relationship between the holomorphic curvature of a continuous complex Finsler metric and the holomorphic curvature of the Carath\'eodory metric and the Kobayashi metric on convex domains. It is a modification of Proposition 1.1 in \cite{Wan}, where the $\partial\bar{\partial}$-operator on the double tangent bundle is used. Our modification uses the generalized Laplacian $\triangle$ (cf. Definition 3.1.5 in \cite{AP}), which is more flexible since it only involves the tangent bundle of the manifold.

\begin{proposition}\label{pmp}
Let $M$ be a complex manifold of complex dimension $n$. Suppose that $F_1$ is an upper semi-continuous complex Finsler metric and $F_2$ is a continuous complex Finsler metric on $M$, respectively. Then for any $(z;[v])\in PT^{1,0}M$, one has
\begin{equation}
\frac{1}{2F_2^2(z;v)}\sup_\varphi\left\{\triangle\left(\log \frac{\varphi^\ast F_1^2}{\varphi^\ast F_2^2}\right)(0)\right\}\geq K_{F_2}(z;[v])-\frac{F_1^2(z;v)}{F_2^2(z;v)}K_{F_1}(z;[v]),\label{cin}
\end{equation}
where the supremum is taken with respect to the family of all holomorphic mappings $\varphi:\mathbb{D}\rightarrow M$ with $\varphi(0)=z,\varphi'(0)=\lambda v$ for some $\lambda\in\mathbb{C}^\ast$.
\end{proposition}

\begin{proof}Indeed, by the definition of Gaussian curvature (cf. Definition 2.5.3 in \cite{AP}) and the definition of generalized Laplacian $\triangle$ (cf. Definition 3.1.5 in \cite{AP}),  we have
\begin{eqnarray*}
\left(K_{F_2}-\frac{F_1^2}{F_2^2}K_{F_1}\right)(z;[v])
&=&\sup_\varphi\left\{K(\varphi^\ast F_2)(0)\right\}-\frac{F_1^2(z;v)}{F_2^2(z;v)}\sup_\varphi\left\{K(\varphi^\ast F_1)(0)\right\}\\
&=&\sup_\varphi\left\{K(\varphi^\ast F_2)(0)\right\}-\sup_\varphi\left\{\frac{F_1^2(z;v)}{F_2^2(z;v)}K(\varphi^\ast F_1)(0)\right\}\\
&\leq&\sup_\varphi\left\{K(\varphi^\ast F_2)(0)-\frac{F_1^2(z;v)}{F_2^2(z;v)}K(\varphi^\ast F_1)(0)\right\}\\
&=&\sup_\varphi\left\{\frac{1}{2\varphi^\ast F_2^2}\triangle \left(\log\frac{\varphi^\ast F_1^2}{\varphi^\ast F_2^2}\right)(0)\right\},
\end{eqnarray*}
where the supermum is taken with respect to the family of all holomorphic mappings $\varphi:\mathbb{D}\rightarrow M$ with $\varphi(0)=z$ and $\varphi'(0)=\lambda v$ for some $\lambda\in\mathbb{C}^\ast$,

\end{proof}

\subsection{A comparison theorem of invariant metrics on convex domains}\label{section-4.2}

It is known that if a complex manifold  admits both the Bergman metric and the Carath\'eodory metric, then these two invariant metrics are comparable.
In general, however, the Bergman metric and the Kobayashi metric are non-comparable.

On the other hand, bounds on the holomorphic sectional curvature allow to compare a complex Finsler metric with the Kobayashi metric and to give conditions for a given complex Finsler metric to be the Kobayashi metric \cite{AP}.

The following theorem \ref{thm-fcc} shows that if an $\mbox{Aut}(\mathfrak{D})$-invariant strongly pseudoconvex complex Finsler metric $F$ on a convex domain $\mathfrak{D}$ is comparable with the Carath\'eodory metric $F_C$ (resp. the Kobayashi metric $F_K$) on $\mathfrak{D}$, then the comparable constants are actually determined by the holomorphic curvatures of $F$ and  $F_C$, respectively. Hence the comparable constants are both analytic and geometric invariants.
\begin{theorem}\label{thm-fcc}
Let $\mathfrak{D}$ be a convex domain in $\mathbb{C}^n$ and $F_C$ be the Carath\'eodory metric on $\mathfrak{D}$.  Suppose that $F:T^{1,0}\mathfrak{D}\rightarrow [0,+\infty)$ is an $\mbox{Aut}(\mathfrak{D})$-invariant strongly pseudoconvex complex Finsler metric  such that
\begin{equation}
c_1F_C^2(z;v)\leq F^2(z;v)\leq c_2F_C^2(z;v),\quad  \forall(z;v)\in T^{1,0}\mathfrak{D}\label{CF}
\end{equation}
for constants $c_1>0,c_2>0$, and the holomorphic sectional curvature $K_F$ of $F$ satisfies
$$
\inf_{(z;[v])\in PT^{1,0}\mathfrak{D}} K_F(z;[v])=-K_1<0\quad \mbox{and}\quad \sup_{(z;[v])\in PT^{1,0}\mathfrak{D}}K_F(z;[v])=-K_2<0
$$
for positive constants $K_1$ and $K_2$, respectively.
 Then it necessary that $c_1\geq \frac{4}{K_1}$ and $c_2\leq \frac{4}{K_2}$, hence
\begin{equation}
\frac{4}{K_1}F_C^2(z;v)\leq F^2(z;v)\leq \frac{4}{K_2}F_C^2(z;v),\quad\forall(z;v)\in T^{1,0}\mathfrak{D}.\label{fc}
\end{equation}
 In particular, if $K_F\equiv -c$ for some positive constant $c>0$, then $F$ coincides with the Carath\'edory metric (hence the Kobayashi metric) on $\mathfrak{D}$.
\end{theorem}

\begin{proof}
It is clear that \eqref{fc} holds for $(z;v)=(z;0)$. It suffices to show that \eqref{fc} holds for all points $(z;v)\in \widetilde{T^{1,0}\mathfrak{D}}$ or equivalently  $(z;[v])\in PT^{1,0}\mathfrak{D}$ since \eqref{fc} is invariant if  $v$ is replaced with $\lambda v$ for any $\lambda\in\mathbb{C}^\ast$.

Suppose at the point $(z_0;[v_0])\in PT^{1,0}\mathfrak{D}$, we have $\frac{F_C(z_0;v_0)}{F(z_0;v_0)}=\frac{1}{c_1}$. Namely $\frac{F_C(z;v)}{F(z;v)}$ achieves its maximum at the point $(z_0;[v_0])\in PT^{1,0}\mathfrak{D}$. Then
for any holomorphic mappings $\varphi:\mathbb{D}\rightarrow \mathfrak{D}$ satisfying $\varphi(0)=z_0$ and $\varphi'(0)=\lambda v_0$ with $\lambda\in\mathbb{C}^\ast$, the function $\log\frac{\varphi^\ast F_C^2}{\varphi^\ast F^2}$ achieves its maximum at the point $\zeta=0\in\mathbb{D}$. Thus
$\triangle\left(\log \frac{\varphi^\ast F_C^2}{\varphi^\ast F^2}\right)(0)\leq 0$,which together with  \eqref{cin} in  Proposition \ref{pmp} implies that
\begin{equation*}
K_{F}(z_0;[v_0])-\frac{F_C^2(z_0;v_0)}{F^2(z_0;v_0)}K_{F_C}(z_0;[v_0])\leq 0.\label{E-a}
\end{equation*}
Thus for any $(z;v)\in PT^{1,0}\mathfrak{D}$,  we have
\begin{equation}
\frac{F_C^2(z;v)}{F^2(z;v)}\leq \frac{1}{c_1}=\frac{F_C^2(z_0;v_0)}{F^2(z_0;v_0)}\leq \frac{K_F(z_0;[v_0])}{-4},\label{ci}
\end{equation}
since on a convex domains we always have $K_{F_C}(z_0;[v_0])\equiv-4$. By assumption, we have
$$-K_1=\inf_{(z;[v])\in PT^{1,0}\mathfrak{D}} K_F(z;[v])\leq K_F(z_0;[v_0])\leq \sup_{(z;[v])\in PT^{1,0}\mathfrak{D}} K_F(z;[v])= -K_2,$$
this together with \eqref{ci} implies $c_1\geq \frac{4}{K_1}$ and $F^2(z;v)\geq \frac{4}{K_1}F_C^2(z;v)$ for any $(z;v)\in \widetilde{T^{1,0}\mathfrak{D}}$.

Note that if $c_1=\frac{4}{K_1}$, then by \eqref{ci} we have $K_F(z_0;[v_0])\leq -K_1$ which implies that $K_F(z_0;[v_0])=-K_1$. Therefore  $K_F$ achieves its infimum  at $(z_0;[v_0])\in PT^{1,0}\mathfrak{D}$ if and only if $\frac{F_C^2}{F^2}$ achieves its maximum at $(z_0;[v_0])\in PT^{1,0}\mathfrak{D}$.

 The proof of the inequality $F^2(z;v)\leq \frac{4}{K_2}F_C^2(z;v)$ for any $(z;v)\in \widetilde{T^{1,0}\mathfrak{D}}$ follows immediately from Theorem 5.1 in \cite{A2}, which can also be derived by Proposition 3.1.14 in \cite{AP}. For completeness, we give the proof here. Indeed, by assumption  $K_F\leq -K_2<0$, it follows that $K_{\frac{\sqrt{K_2}}{2}F}\leq -4$. Thus by (ii) in Proposition 3.1.14 in \cite{AP}, we have
 $
 \frac{\sqrt{K_2}}{2}F\leq F_C
 $. This implies  $F^2\leq \frac{4}{K_2}F_C^2$
 since on any convex domains the Carath\'eodory metric and the Kobayashi metric coincide, and they are continuous complex Finsler metrics with constant holomorphic curvatures $-4$.
\end{proof}

\begin{remark}\label{rhc}
The definition of holomorphic sectional curvature (resp. holomorphic bisectional curvature) adopted in this paper  is equal to  $2$ times that of the holomorphic sectional curvature (resp. holomorphic bisectional curvature) adopted in \cite{Look3}.
If we normalize both $F_C$ and $F$ so that  the holomorphic curvature of $F_C$ is $-1$, and the holomorphic sectional curvature of $F$ is bounded below and above by $-1$ and $-\frac{K_1}{K_2}$, respectively, then \eqref{fc} turns out to be
$$
F_C^2(z;v)\leq F^2(z;v)\leq \frac{K_1}{K_2}F_C^2(z;v).
$$
\end{remark}

By Theorem \ref{m-th-1} and \ref{thm-fcc}, we immediately have the following comparison theorem of invariant strongly pseudoconvex complex Finsler metrics with the Carath\'eodory metric (resp. the Kobayashi metric), which will be used to prove Theorem \ref{m-th-2}.

\begin{theorem}\label{lemb}
Let $\mathfrak{D}$ be a classical domain which is endowed with an $\mbox{Aut}(\mathfrak{D})$-invariant K\"ahler-Berwald metric $F:T^{1,0}\mathfrak{D}\rightarrow [0,+\infty)$. Suppose
$$\min_{(Z;V)\in PT^{1,0}\mathfrak{D}} K_{F}(Z;[V])=-K_1<0\;\mbox{and}\;
 \max_{(Z;V)\in PT^{1,0}\mathfrak{D}} K_{F}(Z;[V])=-K_2<0.
$$
Then
\begin{equation}
\frac{4}{K_1}F_C^2(Z;V)\leq F^2(Z;V)\leq \frac{4}{K_2}F_C^2(Z;V),\quad \forall (Z;V)\in T^{1,0}\mathfrak{D}.\label{fec}
\end{equation}
In particular, if $F_B:T^{1,0}\mathfrak{D}\rightarrow [0,+\infty)$ is the Bergman metric on $\mathfrak{D}$, then we have
\begin{equation}
\frac{4}{K_1}F_C^2(Z;V)\leq F_B^2(Z;V)\leq \frac{4}{K_2}F_C^2(Z;V),\quad  \forall (Z;V)\in T^{1,0}\mathfrak{D}.\label{CBC}
\end{equation}
\end{theorem}

\begin{proof}
It is well-known that a classical domain $\mathfrak{D}$ is a convex domain in $\mathbb{C}^r$ with $r=\dim\,(\mathfrak{D})$. So that the Carath\'eodory metric and the Kobayashi metric on $\mathfrak{D}$ coincide, namely $F_C(Z;V)=F_K(Z;V)$ for any $(Z;V)\in T^{1,0}\mathfrak{D}$. Moreover, both of them are continuous complete complex Finsler metrics with constant holomorphic curvature $-4$.

On the other hand, by Theorem 2 in \cite{Suzuki}, the function
$\phi(V):=\frac{F^2(0;V)}{F_C^2(0;V)}$
is continuous and well-defined on $PT_{0}^{1,0}\mathfrak{D}\cong\mathbb{CP}^{r-1}$ with $r:=\dim\,(\mathfrak{D})$. Since $\mathbb{CP}^{r-1}$ is  compact, there exist two positive constants $c_1$ and $c_2$ which are exactly the minimum and maximum of $\psi$ on $PT_{0}^{1,0}\mathfrak{D}$ respectively, such that
$c_1\leq \phi(V)\leq c_2$ for any $V\in PT_{0}^{1,0}\mathfrak{D}$.
Thus we have
$c_1F_C^2(0;V)\leq F^2(0;V)\leq c_2F_C^2(0;V)$ for any $V\in T_{0}^{1,0}\mathfrak{D}$.
 Since $\mathfrak{D}$ is a homogeneous domain,  $F_C$ and $F$ are  $\mbox{Aut}(\mathfrak{D})$-invariant. This implies that
$$
c_1F_C^2(Z;V)\leq F^2(Z;V)\leq c_2F_C^2(Z;V),\quad \forall (Z;V)\in T^{1,0}\mathfrak{D}.
$$
Thus by Theorem \ref{thm-fcc} and the assumption of the Theorem \ref{lemb}, we immediately obtain \eqref{fec}.
\end{proof}

\begin{remark}
By Remark \ref{rhc} and Theorem 6, 9, 11, 12 and 16 in \cite{Look3}, it is easy  to check that the Bergman metrics on  the classical domains indeed satisfy \eqref{CBC} in Theorem \ref{lemb}.
\end{remark}

\subsection{Proof of Theorem \ref{m-th-2}}

Now let $D\subset\mathbb{C}^n$ be a homogeneous domain which is endowed with an $\mbox{Aut}(D)$-invariant strongly pseudoconvex complex Finsler metric $F:T^{1,0}D\rightarrow [0,+\infty)$.
The following theorem shows that, in order to establish a Schwarz lemma for holomorphic mappings $f$ from $D$ into itself with respect to $F$ which is not necessary Hermitian quadratic, it suffices to establish a Schwarz lemma for the induced complex Minkowski norm $\pmb{f}(v):=F(0;v)$ for $v\in T_0^{1,0}D$.
\begin{theorem}\label{h-thm}
Let $D\subset\subset \mathbb{C}^n$ be a bounded homogeneous domain which  contains the origin $0$ and $F:T^{1,0}D\rightarrow [0,+\infty)$  an $\mbox{Aut}(D)$-invariant complex Finsler metric. If there exists a positive constant $c$ such that for any holomorphic mapping $f:D\rightarrow D$ with $f(0)=0$,
\begin{equation}
(f^\ast F)(0;v)\leq \sqrt{c}F(0;v),\quad \forall v\in T_0^{1,0}D.\label{fF}
\end{equation}
Then for any fixed point $z_0\in D$ and any holomorphic mapping $g:D\rightarrow D$, we have
\begin{equation}
(g^\ast F)(z_0;v)\leq \sqrt{c}F(z_0;v),\quad \forall v\in T_{z_0}^{1,0}D.\label{gF}
\end{equation}
\end{theorem}

\begin{proof}
Let $z_0$ be any fixed point in $D$. Denote $w_0=g(z_0)$. Since $D$ is homogeneous, there exists $h_{z_0}\in\mbox{Aut}(D)$ and $l_{w_0}\in\mbox{Aut}(D)$ such that
$h_{z_0}(z_0)=0$ and $l_{w_0}(w_0)=0$. Denote $h_{z_0}^{-1}$ the inverse of $h_{z_0}$.  Then $h_{z_0}^{-1}\in\mbox{Aut}(D)$ and $h_{z_0}^{-1}(0)=z_0$. Thus
$l_{w_0}\circ g\circ h_{z_0}^{-1}: D\rightarrow D$ is a holomorphic mapping satisfying
$l_{w_0}\circ g\circ h_{z_0}^{-1}(0)=0$.
By assumption \eqref{fF}, we have
$$
F\left(\left(l_{w_0}\circ g\circ h_{z_0}^{-1}\right)(0);\left(l_{w_0}\circ g\circ h_{z_0}^{-1}\right)_\ast (v)\right)\leq \sqrt{c}F(0;v),\quad \forall v\in T_0^{1,0}D.
$$
Since $F$ is $\mbox{Aut}(D)$-invariant and $l_{w_0}\in\mbox{Aut}(D)$, we have
\begin{eqnarray*}
F\left(g(z_0);g_\ast(v)\right)&=&F\left(l_{w_0}\circ g(z_0);(l_{w_0})_\ast g_\ast(v)\right)\\
&=&F\left(\left(l_{w_0}\circ g\circ h_{z_0}^{-1}\right)(0);\left(l_{w_0}\circ g\circ h_{z_0}^{-1}\right)_\ast (h_{z_0})_\ast(v)\right).
\end{eqnarray*}
Again applying \eqref{fF} to  the right hand side of the above last equality we obtain
\begin{eqnarray*}
F\left(g(z_0);g_\ast(v)\right)\leq \sqrt{c}F(z_0;v),
\end{eqnarray*}
which is exactly the inequality \eqref{gF}.
\end{proof}

The following lemma is a direct consequence of Theorem 6.3 in \cite{Kobayashi2}.

\begin{lemma}\label{lem-a}
Let $\mathbb{D}$ be the open unit disk in $\mathbb{C}$ which is endowed with an $\mbox{Aut}(\mathbb{D})$-invariant K\"ahler metric
$\mathcal{P}(z;dz)=\frac{2}{\sqrt{K_1}}\frac{\mid dz\mid}{1-\mid z\mid^2}$ with constant Gaussian curvature $-K_1<0$. Let $\mathfrak{D}$ be a classical domain  which is endowed with an $\mbox{Aut}(\mathfrak{D})$-invariant K\"ahler-Berwald metric $F$ such that its holomorphic sectional curvature is bounded above by a negative constant $-K_2$. Then for any holomorphic mapping $f$ from $\mathbb{D}$ into $\mathfrak{D}$, we have
\begin{equation}
(f^\ast F)(z;v)\leq \sqrt{\frac{K_1}{K_2}}\mathcal{P}(z;v),\quad \forall (z;v)\in T^{1,0}\mathbb{D}.
\end{equation}
\end{lemma}

Now we are in a position to prove Theorem \ref{m-th-2}.

\begin{proof}(\textbf{Proof of Theorem \ref{m-th-2}}). It is clear that \eqref{fhg} holds whenever $V=\pmb{0}$. Thus it suffices to show that \eqref{fhg} holds for any $\pmb{0}\neq V\in T_Z^{1,0}\mathfrak{D}_1$.
By  \eqref{fec} in Theorem \ref{lemb},  for each fixed vector $V\in T_{0}^{1,0}\mathfrak{D}_1$ with $F_1(0;V)=\sqrt{\frac{4}{K_1}}$, we  have
$\sqrt{\frac{4}{K_1}}F_C(0;V)\leq F_1(0;V),$
 which implies that $F_C(0;V) \leq 1$.
Thus $V\in \overline{I_0(F_C)}$, which is the closure of $I_0(F_C):=\{V\in T_0^{1,0}\mathfrak{D}_1:F_C(0;V)<1\}$. Since $\mathfrak{D}_1$ is a convex domain, by Theorem B in \cite{Suzuki} we have $I_0(F_C)=\mathfrak{D}_1$. Using  $V$  we are able to define a holomorphic embedding $j_V:\mathbb{D}\rightarrow \mathfrak{D}_1$. Indeed, let $\mathbb{D}$ be the open unit disk endowed with the $\mbox{Aut}(\mathbb{D})$-invariant metric $\mathcal{P}(z;v)=\frac{2}{\sqrt{K_1}}\frac{\mid v\mid}{1-\mid z\mid^2}$ whose Gaussian curvature is identically $-K_1$. Define $j_V: (\mathbb{D}, \mathcal{P})\rightarrow (\mathfrak{D}_1,F_1)$ as follows:
\begin{equation}
j_V(z)=zV\in \mathfrak{D}_1, \quad z\in \mathbb{D}.\label{emb}
\end{equation}
It is clear that $j_V:\mathbb{D} \rightarrow \mathfrak{D}_1$ is a holomorphic mapping. Furthermore, since
\begin{equation}
(j_V^\ast F_1)(0;1)=F_1(0;V)=\sqrt{\frac{4}{K_1}}=\mathcal{P}\left(0;1\right),\label{ism}
\end{equation}
it follows that $j_V:(\mathbb{D}, \mathcal{P})\rightarrow (\mathfrak{D}_1,F_1)$ is a holomorphic isometry at the origin $0\in\mathfrak{D}_1$.

On the other hand, for each $V\in T_0^{1,0}\mathfrak{D}_1$ satisfying $F_1(0;V)=\sqrt{\frac{4}{K_1}}$, there exists a holomorphic mapping $j_V:\mathbb{D}\rightarrow\mathfrak{D}_1$ and a tangent vector $y=1\in\mathbb{C}\cong T_0^{1,0}\mathbb{D}$ such that $j_V(0)=0$ and $(j_V)_\ast(1)=V$, here $(j_V)_\ast$ denotes the differential of $j_V$ at $0\in\mathbb{D}$. Indeed, it suffices to define $j_V$  by \eqref{emb}. Thus for any holomorphic mapping $f: (\mathfrak{D}_1, F_1) \rightarrow (\mathfrak{D}_2, F_2)$, we have
\begin{eqnarray}
\left(f^\ast F_2\right)(0;V)
&=&F_2\left(f(0);f_{\ast}(V)\right)
=F_2\left((f\circ j_V)(0);(f\circ j_V)_{\ast} (1)\right)\nonumber\\
&\leq& \sqrt{\frac{K_1}{K_2}}\mathcal{P}(0;1)\label{ieqa}\\
&=&\sqrt{\frac{K_1}{K_2}}F_1(0;V),\label{ieqb}
\end{eqnarray}
where the inequality \eqref{ieqa} follows from Lemma \ref{lem-a} (applied to the holomorphic mapping $f \circ j_V: (\mathbb{D}, \mathcal{P}) \rightarrow (\mathfrak{D}_2,F_2)$ ) and the equality \eqref{ieqb} follows from \eqref{ism}. Thus we have proved
\begin{equation}
(f^\ast F_2)(0;V)\leq
\sqrt{\frac{K_1}{K_2}}F_1(0;V),\label{pbva}
\end{equation}
for any holomorphic mapping $f:(\mathfrak{D}_1,F_1)\rightarrow (\mathfrak{D}_2,F_2)$ and any fixed tangent vector $V\in T_0^{1,0}\mathfrak{D}_1$ satisfying $F_1(0;V)=\sqrt{\frac{4}{K_1}}$.

Since for any $\pmb{0}\neq V\in T_0^{1,0}\mathfrak{D}_1$, using the homogeneity of the metric $F_1$,  we have
$$
F_1\left(0;\sqrt{\frac{4}{K_1}}\frac{V}{F_1(0;V)}\right)\equiv\sqrt{\frac{4}{K_1}}
$$
and \eqref{pbva} is invariant whenever we replace $V$ with $\lambda V$ such that $\lambda\in\mathbb{C}^\ast$.
Thus \eqref{pbva}  actually holds for any $V\in T_0^{1,0}\mathfrak{D}_1$. Since $\mathfrak{D}_1$ is a homogeneous domain, composing $f$ with $\varphi_{Z_0}^{-1}\in\mbox{Aut}(\mathfrak{D}_1)$ (here $\varphi_{Z_0}(Z_0)=0$) if necessary, we have
\begin{eqnarray}
\left(f^\ast F_2\right)(Z_0;V)
&=&F_2\left(f(Z_0);f_\ast(V)\right)\nonumber\\
&=&F_2\left(\left(f\circ \varphi_{Z_0}^{-1}\right)(0);\left(f\circ\varphi_{Z_0}^{-1}\right)_\ast\left((\varphi_{Z_0})_\ast (V)\right)\right)\nonumber\\
&=&\left(\left(f\circ \varphi_{Z_0}^{-1}\right)^\ast F_2\right)\left(0;\left(\varphi_{Z_0}\right)_\ast (V)\right)\nonumber\\
&\leq&\sqrt{\frac{K_1}{K_2}} F_1\left(0;\left(\varphi_{Z_0}\right)_\ast (V)\right)\label{U-1}\\
&=&\sqrt{\frac{K_1}{K_2}} F_1\left(\varphi_{Z_0}(Z_0);\left(\varphi_{Z_0}\right)_\ast(V)\right)\nonumber\\
&=&\sqrt{\frac{K_1}{K_2}} F_1(Z_0;V)\label{U-2}
\end{eqnarray}
where the inequality \eqref{U-1} follows from \eqref{ieqb} and we use the $\mbox{Aut}(\mathfrak{D}_1)$-invariant property of $F_1$ to obtain \eqref{U-2} since $\varphi_{Z_0}\in\mbox{Aut}(\mathfrak{D}_1)$.
Thus we have proved the following inequality
\begin{equation}
\left(f^\ast F_2\right)(Z_0;V) \leq \sqrt{\frac{K_1}{K_2}}F_1(Z_0;V)\label{eqz}
\end{equation}
for any fixed $Z_0\in \mathfrak{D}_1$ and any vector $V\in T_{Z_0}^{1,0}\mathfrak{D}_1$. We remark that \eqref{eqz} also follows from Theorem \ref{h-thm}  since we have shown that \eqref{pbva} actually holds for any $V\in T_0^{1,0}\mathfrak{D}_1$.

Finally, since \eqref{eqz} holds for an arbitrary fixed point $Z_0\in \mathfrak{D}_1$, changing $Z_0$ into $Z$ if necessary, we obtain \eqref{fhg}. This completes the proof.

\end{proof}

\vskip0.3cm
\noindent
\textbf{Acknowledgement.}\ {The author is very grateful to the referee for invaluable comments and suggestions for revision which have significantly improved the quality of the paper. This work is supported by the National Natural Science Foundation of China (Grant Nos. 12471080, 12071386)}

\vskip0.5cm
\noindent
\textbf{Declarations}\\
\noindent
\textbf{Conflicts of interest/Competing interests} (There is no conflicts of interest)\\
\noindent
\textbf{Availability of data and material} (My manuscript has no associated data)\\
\noindent
\textbf{Code availability} (Not applicable)\\



\end{document}